\RequirePackage{etex}
\documentclass[a4paper,10pt,twoside]{article}

\usepackage[utf8]{inputenc}

\usepackage{amsmath}  
\usepackage{amssymb}
\usepackage{amsthm}
\usepackage{mathrsfs}  
\usepackage{amsfonts}
\usepackage{mhequ}
\usepackage{verbatim}

\usepackage{mathptmx}
\DeclareMathAlphabet{\mathcal}{OMS}{cmsy}{m}{n} 

\usepackage{graphicx}
\usepackage{color}
\usepackage{xcolor}
\usepackage{tikz}
\usetikzlibrary{calc,intersections,through,backgrounds}
\usetikzlibrary{decorations.pathreplacing}
\usetikzlibrary{shapes}

\usepackage{float} 
\usepackage{enumerate}
\usepackage[margin=1.5in, marginparwidth=1.2in]{geometry}

\usepackage{macros}

\usepackage{setspace,dsfont,bbm}

\usepackage[pdfauthor={F. Kunick, P. Tsatsoulis}, pdftitle={Gradient-type estimates for the dynamic $\varphi^4_2$-model}, pdftex]{hyperref}
\hypersetup{
  colorlinks=false,
  pdfborder={0 0 0}
  }
\usepackage{url} 

\usepackage{mathtools} 

\usepackage[capitalise]{cleveref}
\usepackage{autonum} 

\usepackage{fancyhdr}

\usepackage[toc,page]{appendix}

\usepackage{tocloft}
\theoremstyle{plain}
\newtheorem{thm}{Theorem}[section]
\newtheorem{cor}[thm]{Corollary}
\newtheorem{lem}[thm]{Lemma}
\newtheorem{prop}[thm]{Proposition}

\theoremstyle{definition}

\theoremstyle{remark}
\newtheorem{rem}[thm]{Remark} 

\numberwithin{equation}{section}

\newcommand{\dd}{\mathrm{d}}
\newcommand{\ee}{\mathrm{e}}
\newcommand{\Exp}{\mathds{E}}

\DeclarePairedDelimiter{\abs}{\lvert}{\rvert}
\DeclarePairedDelimiter{\norm}{\lVert}{\rVert}
\DeclarePairedDelimiter{\bra}{(}{)}
\DeclarePairedDelimiter{\pra}{[}{]}
\DeclarePairedDelimiter{\set}{\{}{\}}

\newcommand{\dx}[1]{\mathrm{d} #1}

\newcommand{\cZ}{\ensuremath{\mathcal Z}}

\newcommand{\N}{{\mathbb N}}
\newcommand{\R}{{\mathbb R}}

\newcommand{\T}{{\mathbb T}}
\newcommand{\Z}{{\mathbb Z}}

\newcommand{\eps}{{\varepsilon}}

\newcommand{\al}{\alpha}

\DeclareMathOperator{\supp}{supp}

\renewcommand{\tilde}{\widetilde}
\renewcommand{\bar}{\overline}
\renewcommand{\eps}{\varepsilon}
\newcommand{\dxi}{\frac{\partial}{\partial \xi}}

\newcommand{\E}[1]{\mathds{E}\pra*{ #1}}
\allowdisplaybreaks[4] 

\definecolor{darkpink}{rgb}{1,0,.6}

\title{Gradient-type estimates for the dynamic \texorpdfstring{$\varphi^4_2$}{varphi42}-model}
\author{Florian Kunick, Pavlos Tsatsoulis}
\date{}

\fancypagestyle{plain}{

  \fancyhf{}
  \fancyfoot[L]{\footnotesize \today }
  \fancyfoot[C]{\thepage}
}

\renewenvironment{abstract}
{
\begin{center}
\begin{minipage}{.9\textwidth}\small\textbf{Abstract}.\noindent \smallskip
}
{
\end{minipage}
\end{center}
}

\newenvironment{keywords}
{
\begin{center}
\begin{minipage}{.9\textwidth}\small\textbf{Keywords}:\noindent
}
{
\end{minipage}
\end{center}
}

\newenvironment{msc}
{
\begin{center}
\begin{minipage}{.9\textwidth}\small\textbf{MSC 2020}:\noindent
}
{
\end{minipage}
\end{center}
}

\begin{document}

\maketitle

\begin{abstract} 

We prove gradient bounds for the Markov semigroup of the dynamic $\varphi^4_2$-model
on a torus of fixed size $L>0$. For sufficiently large mass $m>0$ these estimates imply exponential contraction of the 
Markov semigroup. Our method is based on pathwise estimates of the linearized equation. To compensate the lack
of exponential integrability of the stochastic drivers we use a stopping time argument and the strong Markov property 
in the spirit of Cass--Litterer--Lyons \cite{CLL13}. Following the classical approach of Bakry-Émery, as a corollary we
prove a Poincar\'e/spectral gap inequality for the $\varphi^4_2$-measure of sufficiently large mass $m>0$ with almost optimal \emph{carr\'e du champ}.  
\end{abstract}


\begin{msc} 60H17, 47D07
\end{msc}

\tableofcontents 

\section{Introduction}

We consider the dynamic $\varphi^4_2$-model on the torus $\T^2=\R^2/L\Z^2$ of fixed size $L>0$ given by
\begin{equs}\label{eq:dynamic_phi4}
 \begin{cases}
 & (\partial_t - \Delta + m) u = - u^3 + 3 \infty u + \sqrt{2} \xi \quad \text{on} \ \R_{>0} \times \T^2,
 \\
 & u \big|_{t=0} = f,
 \end{cases}
\end{equs}
where $m>0$ is a positive mass, $\xi$ denotes space-time white noise and $f$ is a suitable initial condition. The infinite counter term
$+ 3 \infty u$ on the r.h.s.~of \eqref{eq:dynamic_phi4} is reminiscent of renormalization (see Section~\ref{sec:pre} below) since the SPDE
is singular due to the roughness of $\xi$. 

\medskip

This model serves as a toy example in the stochastic quantization of Euclidean quantum field theories. It describes the natural 
reversible dynamics of the $\varphi^4_2$-measure formally given by  
\begin{equs} \label{eq:phi4_formal}
 \nu(\dd u) = \frac{1}{\cZ} \exp\left\{-\int_{\T^2} \dd x  \left(\frac{1}{2} |\nabla u(x)|^2+ \frac{1}{4} |u(x)|^4 - \frac{3\infty}{2} |u(x)|^2\right)\right\} \dd u. 
\end{equs}
The construction of \eqref{eq:phi4_formal} was one of the first achievements in quantum field theory and goes back to 
Nelson \cite{Ne}. Alternatively, Parisi and Wu in \cite{PW81} proposed the use of \eqref{eq:dynamic_phi4} in order to construct
and sample via MCMC methods the measure \eqref{eq:phi4_formal}. A first attempt to implement this approach was made by Da 
Prato and Debussche in \cite{DPD03}. Later, along with the development of regularity structures \cite{Ha14} and paracontrolled calculus \cite{GIP15}, \eqref{eq:dynamic_phi4} was studied extensively by many authors, see for example \cite{MW17I, RZZ17I, RZZ17II, TW18, HMK18, MW17II, GH19, GH21}. 
These results justified rigorously the connection of the singular dynamics and the measure in the sense of Parisi and Wu.   

In the current work we study the regularization properties of the Markov semigroup $\{P_t\}_{t\geq0}$ associated to \eqref{eq:dynamic_phi4} (see \eqref{eq:markovsem} below for the definition) through gradient-type estimates.   
Gradient-type estimates of Markov semigroups are important in the study of functional inequalities, e.g. spectral gap (or infinite dimensional Poincar\'e) and $\log$-Sobolev inequalities, and transportation inequalities 
(see for example \cite{Ka06, BE85, BGL14, CG14}). These estimates usually require some convexity assumption, see for example \cite[Property (\textbf{H.C.K.}), p. 232 and p. 235]{CG14}. 
In the case of \eqref{eq:dynamic_phi4} convexity is destroyed by the presence of the infinite counter term $-3\infty u$ and at first glance it is unclear whether any type of such estimates can be derived. The argument 
we present here allows us to prove the following gradient estimate for the semigroup $\{P_t\}_{t\geq0}$. 

\begin{thm}\label{thm:weak_grad_est} 
Let $\{P_t\}_{t\geq0}$ be the Markov semigroup associated to \eqref{eq:dynamic_phi4} and $\kappa\in(0,1)$. 
For every $q>1$ and $\eps<1-\kappa$ there exists $m_*\equiv m_*(\eps,q,L)>0$ such that
\begin{equs}\label{eq:weak_grad_est}
 \|D P_tF(f)\|_{L^2_x} \leq C (t\wedge1)^{-\frac{\kappa+\eps}{2}} \ee^{-(m-m_*) t} \left(P_t\|DF\|_{H^{-\kappa}_x}^q(f)\right)^\frac{1}{q}, 
\end{equs}
for every cylindrical functional $F$, $t>0$, $f\in C^{-\al_0}$  and an implicit constant $C\equiv C(\eps,\kappa,q,L)<\infty$ which is uniform in $f$ and $m$. In the case $\kappa=0$ the estimate holds for $\eps=0$ and a universal constant $C$ which is independent of $L$.
\end{thm}

Replacing $L^2_\omega$-norm on the r.h.s.~by an $L^1_\omega$-norm yields the strong gradient estimate \cite[Theorem~3.2.4]{BGL14}. The main difference is that  the strong gradient estimate 
implies the $\log$-Sobolev inequality, while \eqref{eq:weak_grad_est} the (weaker) spectral gap inequality (see for example \cite[Section 1]{CG14} and \cite[Sections~4 and 5]{BGL14}).  
Note that in contrast to the classical literature here we insist on a gradient estimate where the r.h.s.~depends on the 
$H^{-\kappa}_x$-norm, allowing for $\kappa$ arbitrarily close to $1$. This is almost in line with the behaviour of the Gaussian free 
field in dimension $2$ where the \emph{carr\'e du champ} is given by the $H^{-1}_x$-inner product or, equivalently, its \emph{Cameron-Martin space} 
is given by $H^1_x$. As an immediate consequence \eqref{eq:weak_grad_est} implies exponential contraction for $m>m_*$ in the following sense,  
\begin{equs}
 \sup_{\|h\|_{L^2_x}\leq 1}\sup_{\|DF\|_{H^{-\kappa}_x\leq 1}} |P_tF(f+h) - P_tF(f)| \leq C (t\wedge1)^{-\frac{\kappa}{2}-\eps} \ee^{-(m-m_*)t}, \label{eq:expcontraction}
\end{equs}
where the second supremum is taken over all cylindrical functionals $F$.  

\medskip

In recent years gradient-type estimates of the form \eqref{eq:weak_grad_est} have seen a rise in popularity. Starting with the work of Bakry--Émery \cite{BE85} it has become a vast research topic to relate these estimates to lower bounds of the Ricci curvature of the associated manifold. 
Since the interpretation of the heat flow on a manifold as a formal gradient flow with respect to the entropy on the Wasserstein space \cite{O01}, the notion of displacement convexity of the entropy is also closely related to lower bounds of the Ricci curvature \cite{OV00}. This relationship can be associated to exponential contraction of the heat flow with respect to the Wasserstein metric which in our case corresponds to \eqref{eq:expcontraction}. Indeed, in \cite{RS05} it has been shown that in the finite-dimensional case all these notions are equivalent. In the infinite-dimensional setting we, for example, refer to \cite{EKS15}.

\medskip
 
In order to prove \eqref{eq:weak_grad_est} we study the linearized equation 
\begin{equs}\label{eq:linearized_phi4}
 \begin{cases}
 & (\partial_t - \Delta + m) J_{0,t}^fh = -3 \big(u^2 - \infty \big) J_{0,t}^fh \quad \text{on} \ \R_{>0} \times \T^2,
 \\
 & J_{0,t}^fh \big|_{t=0} = h,
 \end{cases}
\end{equs}
for suitable initial condition $h$. In the absence of the counter term one easily obtains a contraction estimate for any $m>0$ of the form
\begin{equs}
 \|J_{0,t}^fh\|_{L^2_x}^2 \leq \ee^{-2mt} \|h\|_{L^2_x}^2, \label{eq:L2_est}
\end{equs}
which in turn implies the strong gradient estimate, see for example \cite[Lemma 2.1]{Ka05} where the same dynamics are considered in the $1$-dimensional setting on the whole space\footnote{Using a post-processing of 
\eqref{eq:L2_est} as in Proposition~\ref{betterboundJ} below one can upgrade the $L^2_x$-estimate to an $H^{-\kappa}_x$-estimate for $\kappa\in[0,1)$ in the case of the torus.}. 
To deal with the counter term we appeal to the Da Prato--Debussche decomposition (see Section~\ref{sec:pre} below), understanding $u^2 - \infty$ as 
\begin{equs}
 u^2 - \infty = v^2 + 2 v \<1> + \<2> + c_{t, \infty}, \label{eq:drift_decomp} 
\end{equs}
where $\<1>$ is the solution to the stochastic heat equation \eqref{eq:she} with zero initial data, $\<2>$ its second Wick power defined in \eqref{stochest} and $c_{t,\infty}$ the constant defined in \eqref{renormconst}\footnote{The
constant $c_{t, \infty}$ appears due to the fact that we insist on using Wick powers of $\<1>$ which at time $t=0$ vanish. This is just a technical convenience but not necessary in our approach.}.
The idea is to treat the lower order terms in \eqref{eq:drift_decomp}, namely  $2 v \<1> + \<2> + c_{t, \infty}$, as drift terms and absorb them to the mass $m$. 
Due to the lack of the required exponential integrability, in order to obtain a meaningful gradient estimate we restart the noise every time the Wick powers exceed a certain barrier
using a stopping time argument in the spirit of Cass--Litterer--Lions \cite{CLL13} for rough differential equations (see Section~\ref{sec:stochasticest} below). This argument allows us to bypass the problem of exponential integrability 
of the Wick powers. Instead, we need to study the exponential integrability of the counting process $N(t)$ of the number of restarts to reach time $t$ which due to the strong Markov property has exponential tails (see Proposition~\ref{expmoments}
below). A crucial ingredient to our approach is the \emph{``coming down from infinity"} property of $v$ first obtained in \cite[Proposition~3.7]{TW18} (see also \cite{MW17II, MW20, GH19} for up-to-date results on ``coming down from infinity"), 
which ensures that the estimates on $N(t)$ do not depend on the initial data $f$, therefore, covering uniformly the whole time interval $[0,\infty)$. As a result of the stopping time argument we prove the following $L^2_x$-estimate for every $p<\infty$, 
\begin{equs}
 \E{\|J_{0,t}^f\|_{L^2_x\to L^2_x}^{p}}^{\frac{1}{p}} \leq C \ee^{-(m-m_*)t}  
\end{equs}
for some $m_*>0$ and $C<\infty$ uniformly in $f$, see Proposition~\ref{boundJ} . 
Using a simple post-processing we can upgrade the above estimate to
\begin{equs}
  \E{\|J_{0,t}^f\|_{L^2_x\to H^{\kappa}_x}^{p}}^{\frac{1}{p}} \leq C (t\wedge1)^{-\frac{\kappa+\eps}{2}} \ee^{-(m-m_*)t},  \label{eq:H1_est}
\end{equs}
see Proposition~\ref{betterboundJ}.

\medskip

As we already mentioned earlier, the motivation to study gradient-type estimates for Markov semigroups comes from applications on functional inequalities. As a consequence of \eqref{eq:weak_grad_est} 
we derive a spectral gap inequality for the Markov semigroup $\{P_t\}_{t\geq0}$ based on the celebrated method of Bakry--Émery. Due to the presence of the $H^{-\kappa}_x$-norm for $\kappa$ arbitrarily close to $1$
the carr\'e du champ is almost optimal when compared to the small scale behaviour of the Gaussian free field in $2$-dimensions on a torus of fixed size $L>0$ (which plays the role of an infra-red cutoff).  

\begin{thm} \label{thm:dynamic_phi4_sg}
Under the assumptions of Theorem~\ref{thm:weak_grad_est} the following spectral gap inequality holds
\begin{align}
 P_tF^2(f) - \big(P_t F(f)\big)^2 \leq C  \int_0^t (s\wedge 1)^{-\kappa-\eps} \ee^{-2(m-m_*)s} \, \dd s \ P_t\|FG\|_{H^{-\kappa}_x}^2(f) \quad \nu\text{-a.s. in} \ f, \label{eq:dynamic_phi4_sg}
\end{align}
for every cylindrical functional $F$, $t>0$ and implicit constant $C\equiv C(\eps,\kappa,L)<\infty$ which is uniform in $f$ and $m$.  In the case $\kappa=0$ the estimate holds for $\eps=0$ and a universal constant $C$ which is independent 
of $L$.
\end{thm}

Let us mention that a spectral gap-type inequality for the Markov semigroup generated by \eqref{eq:dynamic_phi4} has
already been obtained in \cite{TW18} in the total variational norm in $C^{-\alpha_0}$ based on a combination of the strong Feller 
property, a support theorem and the ``coming down from infinity" property. The same holds in dimension $3$ based on 
the results from \cite{HM18, HS21, MW17II}. Although the total variational norm is stronger than any
Wasserstein metric, the results in \cite{TW18} do not provide an estimate w.r.t.~the $L^2_x$-derivative.

\medskip

Using the ergodicity of $P_t$, see for example \cite[Corollary 6.6]{TW18}, as a corollary we prove a spectral gap inequality for 
the $\varphi^4_2$-measure for large masses $m>m_*$. 

\begin{cor}\label{cor:phi4_sg} Under the statement of Theorem~\ref{thm:dynamic_phi4_sg} and the additional assumption $m>m_*$ the $\varphi^4_2$-measure satisfies the spectral gap inequality
\begin{equs}
 \Exp_\nu F^2 - \big(\Exp_\nu F\big)^2 \leq C \frac{1}{(m-m_*)^{1-\kappa-\eps}\wedge (m-m_*)} \Exp_\nu\|DF\|_{H^{-\kappa}_x}^2, \label{eq:phi4_sg}
\end{equs}
for every cylindrical functional $F$, where for $\kappa=0$ the estimate holds for $\eps=0$.
\end{cor}

\begin{rem}
We emphasize that in order to obtain \eqref{eq:phi4_sg} we need to choose $m$ large enough and, in particular, $m>m_*$ to ensure that the spectral gap constant does not blow-up in the
limit $t\nearrow\infty$. This is a technical restriction of the method presented here and it is rather unnatural in the case of the torus. On the other hand, such a condition would be natural in the whole plane regime, 
provided that the dependence of the implicit constant $C$ and the mass $m_*$ on $L$ can be eliminated. As we already stated in Theorem~\ref{thm:weak_grad_est}, $C$ does not depend on $L$ for $\kappa=0$
and it would be interesting to investigate whether the dependence of $m_*$ on $L$ can be eliminated as well to allow for a large scale analysis. At first sight this seems possible using suitable weighted norms 
(in the spirit of \cite{MW17I, GH19}), but it is rather unclear whether one can derive meaningful estimates in this direction. 
\end{rem}     

Spectral gap inequalities are a convenient tool which quantifies ergodicity. When it comes to applications beyond the 
study of long time behavior, they have been used in the context of stochastic homogenization \cite{GO11} to obtain stochastic 
estimates on the corrector. In a similar spirit, spectral gap inequalities can be used as a tool in deriving stochastic estimates
in the context of singular SPDEs \cite{LOTT21} (see also \cite[Section 5]{IORT20} for a simpler example). 
 
\medskip 

While completing this work, a relevant work \cite{BD22} appeared, which derives $\log$-Sobolev inequalities for the 
$\varphi^4$-measure in dimensions $2$ and $3$ with carr\'e du champs given by the $L^2_x$-norm. More precisely, 
the authors study approximations of the measure with ultraviolet and infra-red cutoffs and derive lower and upper bounds
on the $\log$-Sobolev constant independent of the cutoffs. Their approach is based on the machinery developed in \cite{BB21} in combination with correlation inequalities. Although these results are optimal in the large scale regime and they imply the spectral gap inequality, the 
techniques presented here are more appropriate in the small scale regime. 

\subsection{Notation}\label{sec:not}
For $\beta \in \R$ we set $C^{\beta} := B^{\beta}_{\infty, \infty}(\T^2)$ and the corresponding norm is denoted by $\norm*{\cdot}_{\beta}$. The space of arbitrarily smooth functions is accordingly denoted by $C^{\infty}$. 
We set $L^p_x := L^p(\T^2)$ and  $\norm*{\cdot}_{L^p_x}$ for the corresponding norm. Similarly, we use the same notation for $L^2_{t, x} := L^2(\R_{+} \times \T^2)$ and $H^{\al}_x := H^{\al}(\T^2)$. Note that we have $L^p_x = B^{0}_{p, p}(\T^2)$ and $H^{\al}_{x} = B^{\al}_{2, 2}(\T^2)$. The space $\mathcal{F}C^{\infty}_b$ denotes all cylindrical functions, i.e. for a distribution $u$ we have $F \in \mathcal{F}C^{\infty}_b$ if there exists $n \in \N$, $\bar{F} \in C^{\infty}_b\bra*{\R^n}$ and $h_i \in C^{\infty}(\T^2)$ for $i = 1, \dots, n$ such that $F(u) = \bar{F}(u(h_1), \dots, u(h_n))$ where we write $u(h) := \int_{\T^2} u(x) h(x) \dx{x}$ for the natural pairing.
Moreover, $a \wedge b := \min\set*{a, b}$.

\subsection{Outline}
In \cref{sec:pre} we recall the Da Prato--Debussche ansatz for \eqref{eq:dynamic_phi4} including the construction and regularity of the Wick powers. In \cref{sec:sop} we outline the ideas needed in order to prove our main theorem. This includes the $L^2_x$-energy estimate of the solution to the linearized equation \eqref{eq:linearized_phi4}, the stopping time argument that we employ in order to bypass the problem of exponential integrability of the Wick powers and finally the upgrade to an $H^{\kappa}_x$-estimate. In \cref{sec:sgi} we prove the spectral gap inequalities \cref{thm:dynamic_phi4_sg} and \cref{cor:phi4_sg}. In \cref{sec:proofs} we include the intermediate proofs of the $L^2_x$- and $H^{\kappa}_x$-estimate. Finally, in the appendix we gather some auxiliary results which are partially known in the literature.

\subsection*{Acknowledgments}
The authors thank F. Otto for discussions and comments. PT would like to thank the Max-Planck-Institute for Mathematics in the Sciences for its warm hospitality.

\section{General framework}\label{sec:pre}

We denote by $\<1>_{0,t}$ the solution to the stochastic heat equation
\begin{align}\label{eq:she}
   \begin{cases}
    \bra*{\partial_t - \Delta + m} \<1> = \sqrt{2}\xi \quad \text{on} \ \R_{>0} \times \T^2 \\
    \<1>|_{t=0} = 0,
   \end{cases}
\end{align}
which is explicitly given by 
\begin{align}
    \<1>_{0,t}(\varphi) = \sqrt{2} \xi\bra*{\mathbbm{1}_{[0, t)} H_{t - \cdot} \ast \varphi}
\end{align}
for all sufficiently nice test functions $\varphi : \T^2 \to \R$ where $(t, x) \mapsto H_t(x)$ denotes the heat kernel associated with the operator $\bra*{\partial_t - \Delta + m}$.
We also denote by $\<2>_{0,t}$ and $\<3>_{0,t}$ its second and third Wick powers defined as the limits
\begin{equs}
 \<2>_{0,t} := \lim_{\delta\searrow 0} \left(\<1>_{0,t}^{(\delta)}\right)^2 - c_{0,t}^{(\delta)}, \quad  \<3>_{0,t} :=
 \lim_{\delta\searrow 0} \left(\<1>_{0,t}^{(\delta)}\right)^3 - 3 c_{0,t}^{(\delta)} \<1>_{0,t}^{(\delta)}, \label{eq:wickpowers}   
\end{equs}
where $c_{0,t}^{(\delta)} = \Exp\left(\<1>_{0,t}^{(\delta)}(0)\right)^2$, $\delta$ denotes some space mollification and the convergence takes place in $C^{-\alpha}$ for every $\alpha>0$. For simplicity, we write $\<k>_{0,t}$, $k=1,2,3$, to denote the collection of $\<1>_{0,t}$, $\<2>_{0,t}$, $\<3>_{0,t}$. We are only interested in the analytical properties of the Wick powers $\<k>_{0,t}$, $k=1,2,3$, given by the next proposition.

\begin{prop}\label{stochest}
Let $T>0$. For any $k = 1, 2, 3$, $\alpha>0$ and $p < \infty$ we have
\begin{align}\label{eq:stochest}
    \E{\sup_{0 \leq t \leq T} \norm*{\<k>_{0,t}}_{-\al}^p}^{\frac{1}{p}} \leq C 
\end{align}
where the constant $C \equiv C(L, T, \al, p)$ does not depend on $m$, vanishes for $T\searrow 0$ and grows at most polynomially in $T$.
\end{prop}

We postpone the proof of this proposition in the appendix, Section~\ref{sec:stochest}, where we present an alternative argument in the spirit of \cite[Section~5]{IORT20} and \cite{LOTT21} using
the fact that the white noise $\xi$ satisfies a spectral gap inequality. Note that we stress the independence of the constant $C$ on $m$, which allows us to ensure that $m_*$ in 
Theorem~\ref{thm:weak_grad_est} is independent of $m$ (in particular, $\theta$ in Proposition~\ref{expmoments} can be chosen independently of $m$). 

\medskip

We interpret the solution $u$ of \eqref{eq:dynamic_phi4}  using the Da Prato--Debussche decomposition \cite{DPD03}, namely, we define $u_{0,t}:=\<1>_{0,t} + v_{0,t}$, where
\begin{align}\label{eq:remainder}
        \begin{cases}
        \bra*{\partial_t - \Delta + m} v_{0,t} = - v_{0,t}^3 - 3v_t^2 \<1>_{0,t} - 3 v_{0,t} \<2>_{0,t} - \<3>_{0,t} + 3 c_{t, \infty} \bra*{\<1>_{0,t} + v_{0,t}} \\
        v|_{t=0} = f,
    \end{cases}
\end{align}
where $f\in C^{-\alpha_0}$ for $\alpha_0>0$ sufficiently small. Let us remark on the constant $c_{t,\infty}$ which appears on the r.h.s. of \eqref{eq:remainder}. This is due to the fact that we renormalize the Wick powers via time dependent constants in order for them to vanish at time $t=0$, although renormalization on the level of the dynamics is done via a time-independent constant $c_{0,\infty}^{(\delta)}$ to ensure that the resulting Markov 
processes is homogeneous in time. In the limit $\delta \searrow 0$ the difference between the two constants leads to
\begin{equs}\label{renormconst}
 c_{t,\infty} := 2 \int_t^\infty H_{2s}(0) \, \dd s \lesssim t^{-\frac{\beta}{2}}
\end{equs}
for every $\beta>0$. A crucial ingredient that we use in the sequel is the ``coming down from infinity" for the solution $v_{0,t}$ to \eqref{eq:remainder} which we include in the appendix,
\cref{sec:remv}. We refer the reader to \cite{DPD03,MW17I,TW18} for details on the global well-posedness of \eqref{eq:remainder}. 

\medskip

For $t \geq s$ we also consider the restarted processes $\<k>_{s, t}$, $k=1,2,3$ which are defined via the solution to
\begin{align}
    \begin{cases}
        \bra*{\partial_t - \Delta + m}\<1>_{s, t} = \sqrt{2} \xi  \quad \text{on} \ \R_{>0} \times \T^2, \\
        \<1>_{s, t}|_{t=s} = 0, 
    \end{cases}
\end{align}
and respectively via \eqref{eq:wickpowers} with $\<1>_{0,t}^{(\delta)}$ replaced by $\<1>_{s,t}^{(\delta)}$. Note that $\{\<k>_{0,t}\}_{t\geq 0}$ and $\{\<k>_{s,t}\}_{t\geq s}$
are equal in law and $\{\<k>_{s,t}\}_{t\geq s}$ is independent of $\{\<k>_{0,t}\}_{t\in [0,s]}$. 

\medskip

Similarly, we consider $v_{s, t}$ which is defined as the solution to
\begin{align}
    \begin{cases}
        \bra*{\partial_t - \Delta +m}v_{s, t} = - v_{s, t}^3 - 3 v_{s, t}^2 \<1>_{s, t} - 3 v_{s, t} \<2>_{s, t } - \<3>_{s, t} +3 c_{t-s, \infty}\bra*{\<1>_{s, t} + v_{s, t}}, \\
        v_{s, t}|_{t=s} = u_s.
    \end{cases}
\end{align}
Note that all pathwise and stochastic estimates for $\<k>_{0,t}$, $k=1,2,3$, and $v_{0,t}$ extend to $\<k>_{s,t}$, $k=1,2,3$, and $v_{s,t}$. Especially, 
due to the ``coming down from infinity" property 
pathwise estimates on $v_{s,t}$ do not depend on $u_s$. 

\section{Strategy of the proof}\label{sec:sop}

In this section we want to give an outline of the proof of \cref{thm:weak_grad_est}. 
By \cite[Theorem 4.2]{TW18} for $f \in C^{-\al_0}$ we know that $\set*{u^f_{0,t}}_{t \geq 0}$ is a Markov process with $u^f_{0,t}|_{t=0} = f$. In particular, for $F \in \mathcal{F}C^{\infty}_b $ the operator
\begin{align}\label{eq:markovsem}
    P_tF(f) := \E{F(u_{0,t}^f)}
\end{align}
yields a one-parameter semigroup. 
We denote by $D$ the $L^2-$derivative, i.e. we have
\begin{align}\label{uderF}
    D F(u) = \sum_{i=1}^n \partial_i \bar{F}(u(h_1), \dots, u(h_n)) h_i.
\end{align}
The implicit function theorem implies that the map
\begin{align}\label{diffv}
    f \mapsto v^f
\end{align}
is differentiable and for any $h \in C^{\infty}$ it holds that $J^f_{0, t}h := v_{0,t}'(f).h$ is a (mild) solution of the equation 
\begin{align}
    \begin{cases}
        (\partial_t - \Delta + m)J^f_{0, t}h = -3 \bra*{\bra*{v^f_{0, t}}^2 + 2 v^f_{0, t} \<1>_{0, t} + \<2>_{0, t}}J^f_{0, t}h + 3 c_{t, \infty} J^f_{0, t}h \\
        J^f_{0, 0}h = h.
    \end{cases}
\end{align}
For a proof we refer to \cref{diffvproof}. 

By definition, we have $u_{0,t}^f = \<1>_{0,t} + v_{0,t}^f$ and since $\<1>_{0,t}$ does not depend on the initial condition we can conclude that also $f \mapsto u^f$ is differentiable, i.e. there exists $u'(f)=v'(f): X \to Y$ (cf. \cref{diffvproof} for the definition of the function spaces $X$ and $Y$) such that
\begin{align}
    u^{f+h} - u^f - u'(f).h = v^{f+h} - v^{f} - v'(f).h = o(\norm{h}_{C^{-\al_0}}).
\end{align}
Thus we can compute for any $t \geq 0$ and $F \in \mathcal{F}C^{\infty}_b$ using a simple Taylor expansion 
\begin{align}
    P_tF(f+h) - P_tF(f) &= \E{F(u_{0,t}^{f+h}) - F(u_{0,t}^f)} \\
    &= \E{\sum_{i=1}^n \partial_i \bar{F}(u_{0,t}^f(h_1), \dots, u_{0,t}^f(h_n))(u_{0,t}^{f+h}(h_i) - u_{0,t}^f(h_i))} + o(\norm*{h}_{C^{-\al_0}}) \\
    &= \E{\sum_{i=1}^n \partial_i \bar{F}(u_{0,t}^f(h_1), \dots, u_{0,t}^f(h_n))(v_{0,t}'(f).h, h_i)_{L^2_x}} + o(\norm*{h}_{C^{-\al_0}}). 
\end{align}
This shows that $f \mapsto P_tF(f)$ is differentiable and we have
\begin{align}\label{chainrule}
    (P_tF)'(f).h &= \E{\bra*{DF(u_{0,t}^f), v_{0,t}'(f).h}_{L^2_x}} = \E{\bra*{DF(u_{0,t}^f), J^f_{0, t}h}_{L^2_x}}.
\end{align}
Moreover, by \cref{betterboundJ} we see that $(P_tF)'(f) : L^2_x \to \R$ is a bounded linear functional\footnote{more specifically the extended operator initially defined on the dense subspace $C^{\infty}$} and thus there exists $DP_tF(f) \in L^2_x$ such that
\begin{align}
    (P_tF)'(f).h = \int_{\T^2} DP_tF(f)(x) h(x) \dx{x}
\end{align}
and in particular
\begin{align}\label{duality}
    \norm*{DP_tF(f)}_{L^2_x} = \sup_{\norm*{h}_{L^2} \leq 1} \abs*{(P_tF)'(f).h}.
\end{align}
 
By \eqref{chainrule}, \eqref{duality} and the H\"older's inequality in probability for any $\kappa \geq 0$
\begin{align}
    \quad \quad \norm*{DP_t F(f)}_{L^2_x} & = \sup_{h \in C^{\infty}, \norm*{h}_{L^2_x}\leq1} \abs*{(P_tF)'(f).h} \\
    &\leq \E{\norm*{D F(u^{f}_{0,t})}^q_{H^{-\kappa}_x}}^{\frac{1}{q}} \bra*{\sup_{h \in C^{\infty}, \norm{h}_{L^2} \leq 1} \E{\norm*{J^f_{0, t}h}^p_{H^{\kappa}_x}}}^\frac{1}{p} \\
    &= \bra*{P_t\norm*{D F}^q_{H^{-\kappa}_x}(f)}^{\frac{1}{q}} \, \bra*{\sup_{h \in C^{\infty}, \norm{h}_{L^2_x} \leq 1} \E{\norm*{J^f_{0, t}h}^p_{H^{\kappa}_x}}}^{\frac{1}{p}} \label{secondineq}
\end{align}
where in the first line we used that $C^{\infty}$ is dense in $L^2_x$.
Hence, in order to prove \cref{thm:weak_grad_est} we have to estimate 
the quantity
\begin{align}\label{eq:J}
    \bra*{\sup_{h \in C^{\infty}, \norm{h}_{L^2_x} \leq 1} \E{\norm*{J^f_{0, t}h}^p_{H^{\kappa}_x}}}^{\frac{1}{p}}
\end{align}
uniformly in the initial condition $f \in C^{-\al_0}$. In order to do this, we will proceed in three steps. The first step is to prove an $L^2_x$-energy estimate with the drawback that the implicit constant is random and moreover it is not clear that it is integrable. The second step -- which is our core argument -- shows that this constant is indeed integrable and moreover uniformly in the initial condition. The third step is a post-processing from $L^2_x$ to $H^{\kappa}_x$ for any $\kappa <1$.

Before we embark in discussing our intermediate results, let us give the proof of our main theorem. 

\begin{proof}[Proof of \cref{thm:weak_grad_est}]
    For $\kappa\in(0,1)$ the assertion follows from combining \eqref{secondineq} and \cref{betterboundJ} below, which 
    provides an estimate on \eqref{eq:J}. For $k=0$ we apply \cref{boundJ}. 
\end{proof}

\subsection{\texorpdfstring{$L^2_x$}{L 2 x}-energy estimate}
For $s \leq t$ and
$h \in C^{\infty}$ we define $J_{s, t}h$ as the solution to the equation
\begin{align}\label{inieq}
    \begin{cases}
        (\partial_t - \Delta + m)J_{s, t}h = -3 \bra*{v_{s, t}^2 + 2 v_{s, t} \<1>_{s, t} + \<2>_{s, t}}J_{s, t}h + 3 c_{t-s, \infty} J_{s, t}, \\
        J_{s, t}h|_{t=s} = h.
    \end{cases}
\end{align}
In order to ease notation we will suppress the dependence on the initial condition but we will always assume that $v_{s,t}|_{t=s} = u_s^f$.

\medskip
The first step towards bounding \eqref{eq:J} is a standard energy estimate in order to bound the $L_x^2$-norm of $J_{s, t }h$ with respect to the $L_x^2$-norm of $h$. From now on, all proofs are postponed to \cref{sec:proofs}. 

\begin{prop}\label{enest}
    For all $s \leq t' \leq t$, $m > 0$ we have
    \begin{align}
       &\norm*{J_{s, t}h}_{L^2_x}^2 + \int_{t'}^t \ee^{-2m(t-r) + 2 \int_r^t g(s, r') \dx{r'}} \norm*{\nabla J_{s, r}h}_{L^2_x}^2 \dx{r} 
        + \int_{t'}^t \ee^{-2m(t-r) + 2 \int_r^t g(s, r') \dx{r'} } \norm*{v_{s, r} J_{s, r}h}_{L^2_x}^2 \dx{r}
        \\
        & \quad \leq \ee^{-2m(t-t') + 2 \int_{t'}^{t} g(s, r) \dx{r}} \norm*{J_{s, t'}h}_{L^2_x}^2, \label{enesth1}
    \end{align}
    where
    \begin{align}
         g(s,t) & := c \bigg(\norm*{\<1>_{s,t}}_{-\al}^2 + \norm*{\<1>_{s,t}}_{-\al}^{\frac{2}{1+\al}} \norm*{\nabla v_{s,t}}_{\infty}^{\frac{2\al}{1 + \al}} + \norm*{\<1>_{s,t}}^{\frac{2}{1-\al}}_{-\al} 
         \\
         & \quad \quad \quad + \norm*{\<2>_{s,t}}_{-\al} + \norm*{\<2>_{s,t}}_{-\al}^{\frac{2}{2-\al}} + c_{t-s, \infty}\bigg)
    \end{align}
    for some deterministic constant $c\equiv c(\alpha) < \infty$.
    In particular, we have
    \begin{align}\label{enestl2}
        \norm*{J_{s, t}h}_{L^2_x}^2 \leq \ee^{-2m(t-s) + 2 \int_s^t g(s, r) \dx{r}} \norm*{h}_{L^2_x}^2.
    \end{align}
\end{prop}
There are some important things we want to remark concerning \cref{enest}. The first remark is that if it were not for the singular nature and the renormalization procedure involved the \emph{error} term $g$ would be zero and hence we would have a clean energy estimate. The second is that in order to prove \cref{thm:weak_grad_est} with an $L^2_x$-norm on the r.h.s. it is enough to consider \eqref{enestl2} but since our goal is to achieve an $H^{\kappa}_x$-estimate it is crucial to use the additional information coming from \eqref{enesth1}, namely, the estimate on the gradient of $J_{s,r}h$ and the product $v_{s,r}J_{s,r}h$. The last  and most important thing we want to remark makes the bridge to our next section. Notice that by Fernique's theorem the quantity $\|\<1>_{s,t}\|_{-\al}$ in $g(s, t)$ has Gaussian moments, whereas $\|\<2>_{s,t}\|_{-\alpha}$ has only exponential moments. Therefore, the pre-factor on the r.h.s. of \eqref{enesth1} fails to be stochastically integrable. To overcome this problem we appeal to a 
stopping time argument, which we explain in the next section. 

\subsection{Stopping time argument and \texorpdfstring{$L^2_x$}{L 2 x}-estimate}\label{sec:stochasticest}

In order to bypass the issue of integrability of $e^{\int_s^t g(s, r) \dx{r}}$ we appeal to probabilistic arguments inspired by \cite{CLL13}. More precisely, we restart the Wick powers $\<k>$, $k=1,2,3$, each time they exceed a certain barrier. This allows us to replace $g(s, t)$ by a the length of the time interval times a
deterministic constant times a counting processes $N(t)$, see 
\eqref{eq:countproc} below. By choosing the length of the time interval small enough we can ensure the exponential integrabillity of the counting process $N(t)$, see Proposition~\ref{expmoments}. The drawback is the exponential factor $\ee^{m_*t}$ appearing in 
Theorem~\ref{thm:weak_grad_est}. 

\medskip

We define the stopping time
\begin{align}
    \tilde{\tau}_1 := \inf \set*{t \geq 0 : \sup_{k = 1, 2, 3} \norm*{\<k>_{0,t}}_{-\al} \geq \eta}
\end{align}
and for $\theta\in(0,1)$ we set
\begin{align}
    \tau_1 := \tilde{\tau}_1 \wedge \theta.
\end{align}
The value of $\eta\equiv\eta(\alpha,L)$ will be fixed via
\begin{equs}
  \sup_{\theta\in(0,1]}\mathds{P}\bra*{\tilde{\tau}_1 \leq \theta} <\frac{1}{4}. \label{eq:prob_small}
\end{equs}
This is possible due to Markov's inequality, \eqref{eq:stochest} and the fact that $\theta<1$ since 
\begin{equs}
  \mathds{P}\bra*{\tilde{\tau}_1 \leq \theta} 
  \leq \mathds{P}\bra*{\sup_{k=1,2,3} \sup_{t \leq \theta} \norm*{\<k>_{0,t}}_{-\al} \geq \eta}
  \leq \mathds{P}\bra*{\sup_{k=1,2,3} \sup_{t \leq 1} \norm*{\<k>_{0,t}}_{-\al} \geq \eta}
  <\frac{1}{4}. 
\end{equs}

\medskip

We inductively define a sequence of stopping times for $n > 1$ via
\begin{align}
    \tilde{\tau}_n &:= \inf \set*{t \geq \tau_{n-1} : \sup_{k = 1, 2, 3} \norm*{\<k>_{\tau_{n-1}, t}}_{-\al} \geq \eta}, 
\end{align}
where $\<k>_{s, t}$ denotes the process at time $t$ restarted at time $s$, and
\begin{align}
    \tau_n := \tau_{n-1} + \bra*{\tilde{\tau}_n - \tilde{\tau}_{n-1}} \wedge \theta.
\end{align}
Furthermore, we define the standard filtration of $\sigma$-algebras for $t > 0$
\begin{align}
    \mathcal{F}_t := \sigma\bra*{\xi(h): \, h \in L^2_{t, x}, \, \supp h \subset (0, t) \times \T^2}.
\end{align}
We notice that since $\sigma\bra*{\<1>_{0,\cdot \wedge t}} \subset \mathcal{F}_t$ and the process $\<1>_{t, t + \cdot}$ is independent of $\<1>_{0, \cdot \wedge t}$ (cf. \cite[Proposition 2.3]{TW18}), by the strong Markov property for any stopping time $\tau$ the process $\<1>_{\tau, \tau + \cdot}$ is independent of $\mathcal{F}_{\tau}$. Since $\sigma\bra*{\<k>_{0, \cdot \wedge t}} \subset \mathcal{F}_t$, we have that for any $n \geq 1$
\begin{align}
    \tilde{\tau}_n - \tilde{\tau}_{n-1} \ \mathrm{is \ independent \ of} \ \mathcal{F}_{\tilde{\tau}_{n-1}}
\end{align}
and thus 
\begin{align}\label{incind}
    \<k>_{\tau_{n}, \tau_{n + \cdot}} \ \mathrm{is \ independent \ of} \ \mathcal{F}_{\tau_{n}}.
\end{align} 

\medskip

Let $t \leq \tau_1$. By the definition of $\tau_1$ we know that $\norm*{\<k>_{0,t}}_{-\al} < \eta$ for all $k=1,2,3$. Then by \cref{enest} and \cref{gradvest} for any $\eps > 0$ we have that
\begin{align}
    \norm*{J_{0,t}h}^2_{L^2_x} &\leq \ee^{-2mt + 2 \int_0^t g(0,r) \dx r} \norm*{h}^2_{L^2_x} \\
    &\leq \ee^{-2mt + 2c \bra*{\eta^2 + \eta^{\frac{2}{1-\al}} + \eta^{\frac{2}{2-\al}}}t + c \eta^{\frac{2}{1+\al}} \int_0^t r^{\frac{2\al}{1+\al} (-1 - \eps)} \dx r} \norm*{h}^2_{L^2_x} \\
    &\leq \ee^{-2mt + 2c\bra*{t + t^{\frac{1-\al\bra*{1+2\eps}}{1 + \al}}}} \norm*{h}^2_{L^2_x}
\end{align}
for some $c\equiv c(\alpha,\eta) < \infty$.
For $\tau_{n-1} \leq t \leq \tau_n$ we have by \cref{enest} in the same manner
\begin{align}
    \norm*{J_{\tau_{n-1}, t}h }^2_{L^2_x} &\leq \ee^{-2m(t - \tau_{n-1}) + 2c \bra*{ (t - \tau_{n-1}) + (t-\tau_{n-1})^{\frac{1 - \al \bra*{1+2\eps}}{1+\al}}}} \norm*{J_{\tau_{n-2}, \tau_{n-1}}h}^2_{L^2_x} \\ 
\end{align}
and thus by induction we get that
\begin{align}
     \norm*{J_{0, t}h}^2_{L^2_x} &\leq \ee^{-2mt + 2c \bra*{(t - \tau_{n-1})^{\frac{1 - \al \bra*{1+2\eps}}{1+\al}} + \bra*{t-\tau_{n-1}} + \sum_{i =1}^{n-1} (\tau_i - \tau_{i-1})^{\frac{1 - \al \bra*{1+2\eps}}{1+\al}} + \bra*{\tau_i - \tau_{i-1}}}} \norm*{h}^2_{L^2_x}.
\end{align}
From now on we set $\gamma := \frac{1 - \al \bra*{1+2\eps}}{1+\al}$.
By introducing the following counting process
\begin{align}\label{eq:countproc}
    N(t) := \inf \set*{n \geq 1 : \tau_n \geq t}
\end{align}
we furthermore estimate using $\tau_i - \tau_{i-1} \leq \theta$ for any $t \geq 0$
\begin{align}\label{energyestrand}
    \norm*{J_{0, t}h}^2_{L^2_x} &\leq \ee^{-2mt + 2c \theta^{\gamma} N(t) } \norm*{h}^2_{L^2_x}.
\end{align}

\begin{rem} Although we suppressed the dependence on the initial condition $f$ to ease the notation,  
we should also point out that our estimates do not depend $f$. This is possible because of the 
``coming down from infinity" property (cf. \cref{sec:remv}), which allows us to ensure that the gradient estimate in 
\cref{thm:weak_grad_est} is uniform in $f$. 
\end{rem}

The above procedure boils down the problem of estimating $J_{0,t}h$ to showing exponential moment for $N(t)$. 
Since the sequence $\{\tau_n\}_{n \geq 1}$ has independent increments\footnote{at least if conditioned onto $\mathcal{F}_{\tau_{n-1}}$} we can expect this provided we choose $\theta$ small enough. This is the content of the next proposition,
which is in the core of our argument, therefore we present the proof here. 

\begin{prop}\label{expmoments}
    Let $c\equiv c(\alpha,\eta)>0$ as in \eqref{energyestrand}. For all $p \geq 1$ there exists $\theta_0\equiv\theta_0(\alpha,p,\eta)\in(0,1)$ which is independent of $m$ such that for all $\theta \leq \theta_0$ and $t\geq 0$
    \begin{align}
        \E{\ee^{pc\theta^{\gamma} N(t)}}^{\frac{1}{p}} \leq C \ee^{\frac{2 \ln 2}{\theta} t},
    \end{align}
   where $C$ is a universal constant uniform in $L$ and $m$. 
\end{prop}

\begin{proof}
    Let $n \geq 1$. The Markov inequality and \eqref{incind} yield
    \begin{align}
        \mathds{P}\bra*{N(t) \geq n} &= \mathds{P}\bra*{\tau_n \leq t} 
        = \mathds{P}\bra*{\sum_{k = 1}^n \bra*{\tau_{k} - \tau_{k-1}} \leq t}
        = \mathds{P}\bra*{\ee^{-\frac{2\ln2}{\theta} \sum_{k = 1}^n \bra*{\tau_{k} - \tau_{k-1}}} \geq \ee^{-\frac{2\ln2}{\theta} t}} \\
        &\leq \ee^{\frac{2\ln2}{\theta} t} \E{ \ee^{-\frac{2\ln2}{\theta} \sum_{k = 1}^n \bra*{\tau_{k} - \tau_{k-1}}}} 
        = \ee^{\frac{2\ln2}{\theta} t}  \bra*{\E{\ee^{-\frac{2\ln2}{\theta} \tau_1}}}^n.\label{est1}
    \end{align}
    Moreover, we estimate
    \begin{align}\label{est2}
        \E{\ee^{-\frac{2\ln2}{\theta} \tau_1}} &\leq \ee^{-2\ln2} + \mathds{P}\bra*{\tilde{\tau}_1\leq \theta} \leq \frac{1}{4} + \mathds{P}\bra*{\tilde{\tau}_1 \leq \theta}.
    \end{align}
    which combined with \eqref{eq:prob_small} yields 
    \begin{align}
        \E{\ee^{-\frac{2\ln2}{\theta} \tau_1}} \leq \frac{1}{2}.
    \end{align}
    Finally, we have by \eqref{est1} that
    \begin{align}
        \mathds{P}\bra*{N(t) \geq n} &\leq  2^{-n} \ee^{\frac{2\ln2}{\theta} t}
    \end{align}
    and the claim follows by choosing $\theta_0$ small enough such that $\theta^{\gamma}_0 < \frac{\ln2}{c p}$. 
\end{proof}

As an immediate consequence of \cref{expmoments} and \eqref{energyestrand} we obtain the following $L^2_x$-estimate. 

\begin{prop}\label{boundJ} For every $p\geq 1$ there exists $m_*\equiv m_*(\alpha,p,L)>0$ such that for every $t\geq0$,
\begin{equs}
 \E{\norm*{J_{0, t}}^{p}_{L^2_x \to L^2_x}}^{\frac{1}{p}} \leq C \ee^{-\bra*{m - m_*}t}, 
\end{equs}
for some universal constant $C<\infty$ which is uniform in $m$ and $L$. 
\end{prop}

\subsection{Upgrade from \texorpdfstring{$L^2_x$}{L 2 x} to \texorpdfstring{$H^{\kappa}_x$}{H kappa x}}

In this section we upgrade the $L^2_x$-estimate in \cref{boundJ} to an $H^{\kappa}_x$-estimate. 

\medskip

The first step is to post-process \cref{enest} using \eqref{energyestrand}.
\begin{cor}\label{twoest}
    For every $t' \leq t$ we have that
    \begin{align}
        &\norm*{J_{0, t}h}^2_{L^2_x} + \int_{t'}^t \ee^{-2m(t-s)} \norm*{\nabla J_{0, s}h}^2_{L^2_x} \dx{s} + \int_{t'}^t \ee^{-2m(t-s)} \norm*{v_{0, s} J_{0, s}h}^2_{L^2_x} \dx{s} \\
        & \quad \leq \ee^{-2mt} \bra*{\ee^{2c\theta^{\gamma} N(t')} + \int_{t'}^t \ee^{2c\theta^{\gamma}N(s)} g(0, s) \dx{s}} \norm*{h}^2_{L^2_x}. \label{twoest1}
    \end{align}
\end{cor}
We can now upgrade \cref{boundJ} to $H^{\kappa}_x$.
\begin{prop}\label{betterboundJ}
    Let $\kappa\in(0,1)$ and $p\geq 1$. For every $\alpha<\frac{1-\kappa}{5}$ there exists $m_*\equiv m_*(\alpha,p,L)>0$ such that
    \begin{align}
        \E{\norm*{J_{0, t}}^{p}_{L^2_x \to H^{\kappa}_x}}^{\frac{1}{p}} \leq C (t\wedge1)^{-\frac{k+5\alpha}{2}} \ee^{-\bra*{m - m_*}t},
    \end{align}
    for some constant $C\equiv C(p,\alpha,\kappa, L)<\infty$ which is uniform in $f$. 
\end{prop}

Here we need $\kappa < 1$ to ensure the integrability of the exponent when  
$t \searrow 0$. Moreover, we again crucially used the ``coming down from infinity" property that ensures that the bound does not depend on the initial data $f$.

\section{Spectral gap inequalities}\label{sec:sgi}

In this section we give our main application of the gradient estimate \cref{thm:weak_grad_est}.
At the core of the argument lies the celebrated method of Bakry and Émery (cf. \cite{B06, BE85, OV00}) to prove $\log$-Sobolev inequalities as well as spectral gap inequalities. 
\begin{thm}
    Let $\dx \lambda = e^{-\psi}\dx{x}$ be a probability measure on a smooth and flat manifold $M$ such that $\psi \in C^{2}(M)$ and $D^2\psi \geq \rho I$ for some $\rho > 0$. Then $\dx\lambda$ satisfies a $\log$-Sobolev inequality with constant $\rho$.
\end{thm}

Hence by the convexity of the potential it is natural to expect that \eqref{eq:dynamic_phi4} even satisfies a $\log$-Sobolev inequality but due to the singular nature of the equation we are only able to prove a spectral gap inequality. At this point we want to mention that in \cite{Ka06} it was shown that \eqref{eq:dynamic_phi4} does satisfy a $\log$-Sobolev inequality when $d=1$\footnote{and the equation does \emph{not} require any renormalization} with respect to $L^2_x$. In the following we also want to point out how the required renormalization procedure obstructs us from proving an $\log$-Sobolev inequality.
The first step is to show the following identity (cf. \cite[p. 131, (3.1.21)]{BGL14}), the proof of of which can be found in 
the appendix, \cref{sec:BEI}.

\begin{prop}\label{BEtrick}
    The following identity holds for every $t>0$ and $F \in \mathcal{F}C^{\infty}_b$, 
    \begin{align}\label{firstident}
    P_t F^2(f)  - \bra*{P_tF(f)}^2 = 2 \int_0^t P_{t-s}\bra*{\norm*{DP_s F}^2_{L^2_x}}(f) \dx{s} \quad \nu\text{-a.s. in} \  f.
\end{align}
\end{prop} 

We are now in position to prove \cref{thm:dynamic_phi4_sg} and \cref{cor:phi4_sg}.

\begin{proof}[Proof of \cref{thm:dynamic_phi4_sg} and \cref{cor:phi4_sg}] We apply Theorem~\ref{thm:weak_grad_est}
combined with \eqref{firstident} and the fact that $P_t$ is a Markov semigroup yielding  
\begin{align}
    P_t \bra*{F^2} - \bra*{P_t F}^2
    \lesssim  \int_0^t (s\wedge1)^{-\kappa - \eps} \ee^{-2\bra*{m - m_*}s} \, \dx{s} \ P_t\norm*{D F}^2_{H^{-\kappa}_x}.
\end{align}
Finally, choosing $m>m_*$ and noting that
\begin{equs}
 \int_0^\infty (s\wedge1)^{-\kappa - \eps} \ee^{-2\bra*{m - m_*}s} \, \dx{s} \lesssim \frac{1}{(m-m_*)^{1-\kappa-\eps}}\vee
 \frac{1}{m-m_*},
\end{equs}
we appeal to ergodicity (cf. \cite[p. 1241, Corollary 6.6]{TW18}) letting $t\nearrow\infty$ to obtain \eqref{eq:phi4_sg}.
\end{proof}

\section{Proof of intermediate statements}\label{sec:proofs}
In this section we collect the proofs of the intermediate statements missing from the previous section.
\begin{proof}[Proof of \cref{enest}]
    Testing the equation \eqref{inieq} with $J_{s, t}h$ yields
    \begin{align}
       & \frac{1}{2} \partial_t \norm*{J_{s, t}h}^2_{L^2_x} + \norm*{\nabla J_{s, t}h}^2_{L^2_x} + m \norm*{J_{s, t}h}^2_{L^2_x} + 3 \norm*{v_{s,t} \bra*{J_{s,t}h}}^2_{L^2_x}
        \\
        & \quad = -6 \bra*{\<1>_{s,t}, v_{s,t} \bra*{J_{s,t}h}^2}_{L^2_x} 
        - 3 \bra*{\<2>_{s,t}, \bra*{J_{s,t}h}^2}_{L^2_x} 
        + 3 c_{t-s, \infty} \norm*{J_{s, t}h}^2_{L^2_x}. \label{energyeq}
    \end{align}
    We start by estimating $\abs*{\bra*{\<1>_{s,t}, v_{s,t} \bra*{J_{s,t}h}^2}_{L^2_x}}$.
    To this end, we apply \cite[Proposition A.8]{TW18} to get
    \begin{align}
        &\abs*{\bra*{\<1>_{s,t}, v_{s,t} \bra*{J_{s,t}h}^2}_{L^2_x}} \lesssim  \norm*{\<1>_{s,t}}_{-\al} \norm*{v_{s,t} \bra*{J_{s,t}h}^2}_{B_{1,1}^{\al}}
    \end{align}
    and then use Proposition A.9 in \cite{TW18} such that we end up with
    \begin{align}
        \norm*{v_{s,t} \bra*{J_{s,t}h}^2}_{B_{1,1}^{\al}} \lesssim \norm*{ v_{s,t}\bra*{J_{s,t}h}^2}^{1- \al}_{L^1_x} \norm*{\nabla\bra*{ v_{s,t}\bra*{J_{s,t}h}^2}}^{\al}_{L^1_x} + \norm*{ v_{s,t}\bra*{J_{s,t}h}^2}_{L^1_x}.
    \end{align}
    Moreover, the Cauchy--Schwarz inequality and the chain rule yield
    \begin{align}
        &\norm*{ v_{s,t}\bra*{J_{s,t}h}^2}^{1- \al}_{L^1_x} \norm*{\nabla\bra*{ v_{s,t}\bra*{J_{s,t}h}^2}}^{\al}_{L^1_x} + \norm*{ v_{s,t}\bra*{J_{s,t}h}^2}_{L^1_x} \\
        & \quad\lesssim \norm*{J_{s, t}h}^{1-\al}_{L^2_x} \norm*{ v_{s,t}\bra*{J_{s,t}h}}^{1-\al}_{L^2_x} \norm*{\bra*{J_{s,t}h}^2 \nabla v_{s,t} + 2 v_{s,t} \bra*{J_{s,t}h} \nabla J_{s, t}h}^{\al}_{L^1_x} + \norm*{J_{s, t}h}_{L^2_x} \norm*{ v_{s,t}\bra*{J_{s,t}h}}_{L^2_x}.
    \end{align}
    The Cauchy--Schwarz inequality again implies
    \begin{align}
        \norm*{\bra*{J_{s,t}h}^2 \nabla v_{s,t} + 2 v_{s,t} \bra*{J_{s,t}h} \nabla J_{s, t}h}^{\al}_{L^1_x} \lesssim \norm*{\nabla v_{s,t}}^{\al}_{L^\infty_x} \norm*{J_{s, t}h}_{L^2_x}^{2\al} + 2^{\al} \norm*{v_{s,t} \bra*{J_{s,t}h}}^{\al}_{L^2_x} \norm*{\nabla J_{s, t}h}^{\al}_{L^2_x}.
    \end{align}
    Hence we have shown that
    \begin{align}
        \abs*{\bra*{\<1>_{s,t}, v_{s,t} \bra*{J_{s,t}h}^2}_{L^2_x}} &\lesssim \norm*{\<1>_{s,t}}_{-\al} \norm*{\nabla v_{s,t}}^{\al}_{L^\infty_x} \norm*{ v_{s,t}\bra*{J_{s, t}h}}^{1-\al}_{L^2_x} \norm*{J_{s, t}h}^{1+\al}_{L^2_x} \\
        & \quad+ \norm*{\<1>_{s,t}}_{-\al} \norm*{J_{s, t}h}^{1-\al}_{L^2_x} \norm*{ v_{s,t}\bra*{J_{s,t}h}}_{L^2_x} \norm*{\nabla J_{s, t}h}^{\al}_{L^2_x} \\
        & \quad + \norm*{\<1>_{s,t}}_{-\al} \norm*{J_{s, t}h}_{L^2_x} \norm*{ v_{s,t}\bra*{J_{s,t}h}}_{L^2_x} \\
        & \quad = I_1 + I_2 + I_3.
    \end{align}
    Then we have by Young's inequality for some $\lambda > 0$ to be chosen later
    \begin{align}
        I_3 = \norm*{\<1>_{s,t}}_{-\al} \norm*{J_{s, t}h}_{L^2_x} \norm*{ v_{s,t}\bra*{J_{s,t}h}}_{L^2_x} \leq \frac{1}{2 \lambda} \norm*{\<1>_{s,t}}_{-\al}^2 \norm*{J_{s, t}h}_{L^2_x}^2 + \frac{\lambda}{2} \norm*{ v_{s,t}\bra*{J_{s,t}h}}^2_{L^2_x}
    \end{align}
    and
    \begin{align}
        I_1 &= \norm*{\<1>_{s,t}}_{-\al} \norm*{\nabla v_{s,t}}^{\al}_{L^\infty_x} \norm*{ v_{s,t}\bra*{J_{s,t}h}}^{1-\al}_{L^2_x} \norm*{J_{s, t}h}^{1+\al}_{L^2_x} \\
        &\leq \frac{1 + \al}{2\lambda^{\frac{1 + \al}{2}}} \norm*{\<1>_{s,t}}_{-\al}^{\frac{2}{1+\al}} \norm*{\nabla v_{s,t}}_{L^\infty_x}^{\frac{2\al}{1 + \al}} \norm*{J_{s, t}h}_{L^2_x}^2 + \frac{(1-\al)\lambda^{\frac{2}{1-\al}}}{2} \norm*{ v_{s,t}\bra*{J_{s,t}h}}^2_{L^2_x} 
    \end{align}
    as well as
    \begin{align}
        I_2 &= \norm*{\<1>_{s,t}}_{-\al} \norm*{J_{s, t}h}^{1-\al}_{L^2_x} \norm*{ v_{s,t}\bra*{J_{s,t}h}}_{L^2_x} \norm*{\nabla J_{s, t}h}^{\al}_{L^2_x} \\
        &\leq \frac{1}{2\lambda} \norm*{\<1>_{s,t}}^2_{-\al} \norm*{J_{s, t}h}_{L^2_x}^{2(1-\al)} \norm*{\nabla J_{s, t}h}^{2\al}_{L^2_x} + \frac{\lambda}{2} \norm*{v_{s,t}\bra*{J_{s,t}h} }^2_{L^2_x} \\
        &\leq \frac{1-\al}{2^{\frac{1}{1-\al}} \lambda^{\frac{2}{1-\al}}} \norm*{\<1>_{s,t}}^{\frac{2}{1-\al}}_{-\al} \norm{J_{s, t}h}^{2}_{L^2_x} + \al \lambda^{\frac{1}{\al}} \norm*{\nabla J_{s, t}h}^2_{L^2_x} + \frac{\lambda}{2} \norm*{ v_{s,t}\bra*{J_{s,t}h}}^2_{L^2_x}.
    \end{align}
    For the second term on the right hand side of \eqref{energyeq} we proceed similarly. First of all, Proposition A.8 in \cite{TW18} yields
    \begin{align}
        \abs*{\bra*{\<2>_{s,t}, \bra*{J_{s,t}h}^2}_{L^2_x}} &\leq \norm*{\<2>_{s,t}}_{-\al} \bra*{\norm*{J_{s, t}h}^{2(1-\al)}_{L^2_x} \norm*{2\bra*{J_{s,t}h} \nabla J_{s, t}h}^{\al}_{L^1_x} + \norm*{J_{s, t}h}^2_{L^2_x}} 
    \end{align}
    and hence the Cauchy--Schwarz inequality combined with Young's inequality with the same $\lambda > 0$ as before yields
    \begin{align}
        \abs*{\bra*{\<2>_{s,t}, \bra*{J_{s,t}h}^2}_{L^2_x}} 
        &\leq 2^{\al} \frac{2-\al}{2 \lambda^{\frac{2}{2-\al}}} \norm*{\<2>_{s,t}}_{-\al}^{\frac{2}{2-\al}} \norm*{J_{s, t}h}_{L^2_x}^2 + \frac{\al \lambda^{\frac{2}{\al}}}{2} \norm*{\nabla J_{s, t}h}^2_{L^2_x}
        + \norm*{\<2>_{s,t}}_{-\al} \norm*{J_{s, t}h}^2_{L^2_x}.
    \end{align}
    Then we set
    \begin{align}
        g(s, t) &:=  \frac{1}{2 \lambda} \norm*{\<1>_{s,t}}_{-\al}^2  + \frac{1 + \al}{2\lambda^{\frac{1 + \al}{2}}} \norm*{\<1>_{s,t}}_{-\al}^{\frac{2}{1+\al}} \norm*{\nabla v_{s,t}}_{L^\infty_x}^{\frac{2\al}{1 + \al}} + \frac{1-\al}{2^{\frac{1}{1-\al}} \lambda^{\frac{2}{1-\al}}} \norm*{\<1>_{s,t}}^{\frac{2}{1-\al}}_{-\al} 
         \\
         & \quad+ \norm*{\<2>_{s,t}}_{-\al} + 2^{\al} \frac{2-\al}{2 \lambda^{\frac{2}{2-\al}}} \norm*{\<2>_{s,t}}_{-\al}^{\frac{2}{2-\al}} 
         + 3c_{t-s, \infty} .
    \end{align}
    By choosing $\lambda$ small enough, we can absorb some of the terms into $\norm*{\nabla J_{s,t}h}_{L^2_x}^2$ respectively $\norm*{v_{s, t} \bra*{J_{s, t}h}}_{L^2_x}^2$ into the right hand side and we end up with the estimate
    \begin{align}\label{energyeq2}
        &\frac{1}{2} \partial_t \norm*{J_{s, t}h}^2_{L^2_x} + m \norm*{J_{s, t}h}^2_{L^2_x} + \frac{1}{2} \norm*{\nabla J_{s, t}h}_{L^2_x}^2 + \frac{1}{2} \norm*{v_{s, t} \bra*{J_{s, t}h}}^2_{L^2_x} \leq g(s,t) \norm*{J_{s, t}h}^2_{L^2_x}.
    \end{align}
   Then the chain rule combined with \eqref{energyeq2} yields 
    \begin{align}
        &\partial_t \bra*{ \ee^{2mt - 2\int_0^t g(s, r) \dx{r}} \norm*{J_{s, t}h}_{L^2_x}^2}  
        + \ee^{2mt - 2\int_0^t g(s, r) \dx{r}} \norm*{\nabla J_{s, t}h}_{L^2_x}^2 + \ee^{2mt - 2\int_0^t g(s, r) \dx{r}} \norm*{v_{s, t} \bra*{J_{s, t}h}}^2_{L^2_x} \leq 0. \label{energyeq3}
    \end{align}
    Integrating \eqref{energyeq3} from $t'$ to $t$ we end up with
    \begin{align}
        &\norm*{J_{s, t}h}_{L^2_x}^2 + \int_{t'}^t \ee^{-2m(t-r) + 2 \int_r^t g(s, r') \dx{r'}} \norm*{\nabla J_{s, r}h}_{L^2_x}^2 \dx{r} 
        + \int_{t'}^t \ee^{-2m(t-r) + 2 \int_r^t g(s, r') \dx{r'} } \norm*{v_{s, t} \bra*{J_{s, r}h}}_{L^2_x}^2 \dx{r} 
        \\
        & \quad \leq \ee^{-2m (t-t') + 2 \int_{t'}^{t} g(s, r) \dx{r}} \norm*{J_{s, t'}h}^2_{L^2_x}.
    \end{align}
\end{proof}

\begin{proof}[Proof of \cref{twoest}]
    The estimate \eqref{energyeq2} yields for $s = 0$
    \begin{align}
        \partial_t \bra*{\ee^{2mt} \norm*{J_{0, t}h}_{L^2_x}^2 } + \ee^{2mt} \bra*{\norm*{\nabla J_{0, t}}_{L^2_x}^2 + \norm*{v_{0, t} J_{0, t}h}_{L^2_x}^2} \leq 2 \ee^{2mt}g(0, t) \norm*{J_{0, t}h}_{L^2_x}^2.
    \end{align}
    Then we integrate from $t'$ to $t$ to obtain
    \begin{align}
        &\norm*{J_{0, t}h}_{L^2_x}^2  + \int_{t'}^t \ee^{-2m\bra*{t-s}} \bra*{\norm*{\nabla J_{0, s}}_{L^2_x}^2 + \norm*{v_{0, s} J_{0, s}h}_{L^2_x}^2} \dx{s} \\ 
        & \quad \leq \ee^{-2m\bra*{t-t'}} \norm*{J_{0, t'}h}^2_{L^2_x} + 2 \int_{t'}^t \ee^{-2m\bra*{t-s}} g(0, s) \norm*{J_{0, s}h}_{L^2_x}^2 \dx{s}.
    \end{align}
    Applying \eqref{energyestrand} to $\norm*{J_{0, t'}h}^2_{L^2_x}$ respectively to $\norm*{J_{0, s}h}_{L^2_x}^2$ yields the assertion.
\end{proof}
\begin{proof}[Proof of \cref{betterboundJ}]
    First of all, Duhamel's formula yields
    \begin{align}
        J_{0, t}h &= S_{\frac{t}{2}} J_{0, \frac{t}{2}}h - 3 \int_{\frac{t}{2}}^t S_{t-s} \set*{ \bra*{v_{0, s}^2 + 2 v_{0, s} \<1>_{0, s} + \<2>_{0, s} - c_{s, \infty}}J_{0, s}h} \dx{s} \\
        & =I_1 + I_2 + I_3 + I_4 + I_5.
    \end{align}
    Then we estimate $I_1$ according to
    \begin{align}
        \norm*{S_{\frac{t}{2}} J_{0, \frac{t}{2}}h}_{\mathcal{B}^{\kappa}_{2,2}}  &\stackrel{\eqref{besovsemi}}{\lesssim} \bra*{t \wedge 1}^{-\frac{\kappa}{2}} \ee^{-m\frac{t}{2}} \norm*{J_{0, \frac{t}{2}}h}_{L^2_x} 
        \stackrel{\eqref{energyestrand}}{\lesssim} \bra*{t \wedge 1}^{-\frac{\kappa}{2}} \ee^{-mt + C \theta^{\gamma}N\bra*{\frac{t}{2}}} \norm*{h}_{L^2_x}.
    \end{align}
   For $I_2$ we further estimate
    \begin{align}
        &\int_{\frac{t}{2}}^t \norm*{S_{t-s} \bra*{v_{0, s}^2 J_{0, s}h}}_{\mathcal{B}^{\kappa}_{2,2}} \dx{s}  \stackrel{\eqref{besovsemi}}{\lesssim} \int_{\frac{t}{2}}^t \bra*{\bra*{t-s} \wedge 1}^{-\frac{\kappa}{2}} \ee^{-m(t-s)} \norm*{v^2_{0, s} J_{0, s}h}_{L^2_x} \dx{s} \\
        &\lesssim
        \int_{\frac{t}{2}}^t \bra*{\bra*{t-s} \wedge 1}^{-\frac{\kappa}{2}} \norm*{v_{0, s}}_{L^\infty_x} \ee^{-m(t-s)} \norm*{v_{0, s} J_{0, s}h}_{L^2_x} \dx{s} \\
        &\lesssim \bra*{\int_{\frac{t}{2}}^t \bra*{\bra*{t-s} \wedge 1}^{-\kappa} \norm*{v_{0, s}}^2_{L^\infty_x} \dx{s}}^{\frac{1}{2}} \bra*{\int_{\frac{t}{2}}^t \ee^{-2m(t-s)} \norm*{v_{0, s} J_{0, s}h}^2_{L^2_x}  \dx{s}}^{\frac{1}{2}} \\
        &\stackrel{\eqref{twoest1}}{\lesssim} \bra*{\int_{\frac{t}{2}}^t \bra*{\bra*{t-s} \wedge 1}^{-\kappa} \norm*{v_{0, s}}^2_{L^\infty_x}\dx{s}}^{\frac{1}{2}} 
         \ee^{-mt} \bra*{e^{2c\theta^{\gamma} N(\frac{t}{2})} + \int_{\frac{t}{2}}^t \ee^{2c\theta^{\gamma}N(s)} g(0, s) \dx{s}}^{\frac{1}{2}} \norm*{h}_{L^2_x}, 
    \end{align}
    where we used again H\"older's inequality in the third step.
    
    Estimating $I_3$ yields
    \begin{align}
        &\int_{\frac{t}{2}}^t \norm*{S_{t-s} \bra*{v_{0, s} \<1>_{0, s} J_{0, s}h}}_{\mathcal{B}^{\kappa}_{2,2}} \dx{s} \stackrel{\eqref{besovsemi}}{\lesssim} \int_{\frac{t}{2}}^t \bra*{\bra*{t-s} \wedge 1}^{-\frac{\kappa + \al}{2}}  \ee^{-m(t-s)} \norm*{v_{0, s} \<1>_{0,s} J_{0, s}h}_{\mathcal{B}^{-\al}_{2,2}} \dx{s} \\
        &\stackrel{\eqref{besovmult1}, \eqref{besovmult2}, \eqref{besovmono}}{\lesssim} \int_{\frac{t}{2}}^t \bra*{\bra*{t-s} \wedge 1}^{-\frac{\kappa + \al}{2}} \norm{\<1>_{0, s}}_{-\al} \norm{v_{0, s}}_{2\al} \ee^{-m(t-s)} \norm{J_{0, s}h}_{\mathcal{B}^{1}_{2,2}} \dx{s} \\
        &\lesssim \bra*{\int_{\frac{t}{2}}^t \bra*{\bra*{t-s} \wedge 1}^{-\kappa - \al} \norm{\<1>_{0, s}}^2_{-\al} \norm{v_{0, s}}^2_{2\al} \dx{s}}^{\frac{1}{2}} \bra*{\int_{\frac{t}{2}}^t \ee^{-2m(t-s)} \norm{J_{0, s}h}^2_{\mathcal{B}^{1}_{2,2}} \dx{s}}^{\frac{1}{2}} \\
        &\stackrel{\eqref{twoest1}}{\lesssim} \bra*{\int_{\frac{t}{2}}^t \bra*{\bra*{t-s} \wedge 1}^{-\kappa - \al} \norm{\<1>_{0, s}}^2_{-\al} \norm{v_{0, s}}^2_{2\al} \dx{s}}^{\frac{1}{2}} \\
        & \quad \times \ee^{-mt} \bra*{\ee^{2c\theta^{\gamma} N(\frac{t}{2})} + \int_{\frac{t}{2}}^t \ee^{2c\theta^{\gamma}N(s)} g(0, s) \dx{s}}^{\frac{1}{2}} \norm*{h}_{L^2_x} \\
        &\lesssim \sup_{0 \leq s \leq t} \norm*{\<1>_{0, s}}_{-\al} \bra*{\int_{\frac{t}{2}}^t \bra*{\bra*{t-s} \wedge 1}^{-\kappa - \al}  \norm{v_{0, s}}^2_{2\al} \dx{s}}^{\frac{1}{2}} \\
        & \quad \times \ee^{-mt} \bra*{\ee^{2c\theta^{\gamma} N(\frac{t}{2})} + \int_{\frac{t}{2}}^t \ee^{2c\theta^{\gamma}N(s)} g(0, s) \dx{s}}^{\frac{1}{2}} \norm*{h}_{L^2_x},
    \end{align}
    using H\"older's inequality in the third step.
    
    The term $I_4$ is estimated via
    \begin{align}
        &\int_{\frac{t}{2}}^t \norm*{S_{t-s} \bra*{\<2>_{0, s} J_{0, s}h}}_{\mathcal{B}^{\kappa}_{2,2}} \dx{s} \stackrel{\eqref{besovsemi}}{\lesssim} \int_{\frac{t}{2}}^t \bra*{\bra*{t-s} \wedge 1}^{-\frac{\kappa + \al}{2}}  \ee^{-m(t-s)} \norm*{\<2>_{0,s} J_{0, s}h}_{\mathcal{B}^{-\al}_{2,2}} \dx{s} \\
        &\stackrel{\eqref{besovmult2}, \eqref{besovmono}}{\lesssim} \int_{\frac{t}{2}}^t \bra*{\bra*{t-s} \wedge 1}^{-\frac{\kappa + \al}{2}} \norm*{\<2>_{0,s}}_{-\al}  \ee^{-m(t-s)} \norm*{J_{0, s}h}_{\mathcal{B}^{1}_{2,2}} \dx{s} \\
        &\lesssim \bra*{\int_{\frac{t}{2}}^t \bra*{\bra*{t-s} \wedge 1}^{-\kappa - \al} \norm*{\<2>_{0,s}}^2_{-\al} \dx{s}}^{\frac{1}{2}} \bra*{\int_{\frac{t}{2}}^t \ee^{-2m(t-s)} \norm*{J_{0, s}h}^2_{\mathcal{B}^{1}_{2,2}} \dx{s}}^{\frac{1}{2}} \\
        &\stackrel{\eqref{twoest1}}{\lesssim} \bra*{\int_{\frac{t}{2}}^t \bra*{\bra*{t-s} \wedge 1}^{-\kappa- \al} \norm*{\<2>_{0,s}}^2_{-\al} \dx{s}}^{\frac{1}{2}}  
        \ee^{-mt} \bra*{\ee^{2c\theta^{\gamma} N(\frac{t}{2})} + \int_{\frac{t}{2}}^t \ee^{2c\theta^{\gamma}N(s)} g(0, s) \dx{s}}^{\frac{1}{2}} \norm*{h}_{L^2_x} \\
        &\lesssim \sup_{0 \leq s \leq t} \norm*{\<2>_{0, s}}_{-\al} \bra*{\int_{\frac{t}{2}}^t \bra*{\bra*{t-s} \wedge 1}^{-\kappa - \al} \dx{s}}^{\frac{1}{2}} 
        \ee^{-mt} \bra*{\ee^{2c\theta^{\gamma} N(\frac{t}{2})} + \int_{\frac{t}{2}}^t \ee^{2c\theta^{\gamma}N(s)} g(0, s) \dx{s}}^{\frac{1}{2}} \norm*{h}_{L^2_x},
    \end{align}
    where again we have used H\"older's inequality in the third step.
    
    Finally, we estimate $I_5$
    \begin{align}
        &\int_{\frac{t}{2}}^t \norm*{S_{t-s} c_{s, \infty} J_{0, s}h}_{\mathcal{B}^{\kappa}_{2,2}} \dx{s} \stackrel{\eqref{besovsemi}}{\lesssim} \int_{\frac{t}{2}}^t  s^{-\frac{\beta}{2}}  \ee^{-m(t-s)} \norm*{J_{0, s}h}_{B^1_{2,2}} \dx{s} \\
        &\lesssim \bra*{\int_{\frac{t}{2}}^t s^{-\beta} \dx{s}}^{\frac{1}{2}} \bra*{\int_{\frac{t}{2}}^t \ee^{-2m(t-s)} \norm*{J_{0, s}h}^2_{L^2_x} \dx{s}}^{\frac{1}{2}} \\
        &\stackrel{\eqref{twoest1}}{\lesssim} \bra*{\int_{\frac{t}{2}}^t  s^{-\beta} \dx{s}}^{\frac{1}{2}} 
        \ee^{-mt} \bra*{\ee^{2c\theta^{\gamma} N(\frac{t}{2})} + \int_{\frac{t}{2}}^t \ee^{2c\theta^{\gamma}N(s)} g(0, s) \dx{s}}^{\frac{1}{2}} \norm*{h}_{L^2_x},
    \end{align}
    where we have used H\"older's inequality in the second step.
    
    Then we use monotonicity of $t \mapsto N(t)$ to infer
    \begin{align}
        \int_{\frac{t}{2}}^t \ee^{2c\theta^{\gamma}N(s)} g(0, s) \dx{s} \leq \ee^{2c\theta^{\gamma}N(t)} \int_{\frac{t}{2}}^t g(0, s) \dx{s} 
    \end{align}
    which all in all yields
    \begin{align}
        \norm*{J_{0, t}h}_{\mathcal{B}^{\kappa}_{2, 2}}  
        & \lesssim \ee^{-mt + C \theta^{\gamma} N(t)} \bra*{1 + \int_{\frac{t}{2}}^t g(0, s) \dx{s}} \norm*{h}_{L^2_x}
        \\
        & \quad \times \bigg(\bra*{t \wedge 1}^{-\frac{\kappa}{2}} + 
        \bra*{\int_{\frac{t}{2}}^t \bra*{\bra*{t-s} \wedge 1}^{-\kappa} \norm*{v_{0, s}}^2_{L^\infty_x}\dx{s}}^{\frac{1}{2}} 
        \\
        & \quad \quad + \sup_{0 \leq s \leq t} \norm*{\<1>_{0, s}}_{-\al} \bra*{\int_{\frac{t}{2}}^t \bra*{\bra*{t-s} \wedge 1}^{-\kappa - \al}  \norm{v_{0, s}}^2_{2\al} \dx{s}}^{\frac{1}{2}} 
        \\
        & \quad \quad + \sup_{0 \leq s \leq t}  \norm*{\<2>_{0, s}}_{-\al} \bra*{\int_{\frac{t}{2}}^t \bra*{\bra*{t-s} \wedge 1}^{-\kappa - \al} \dx{s}}^{\frac{1}{2}} + \bra*{\int_{\frac{t}{2}}^t  s^{-\beta} \dx{s}}^{\frac{1}{2}}\bigg).
    \end{align}
    Using the definition of $g$ we see that
    \begin{align}
        \int_{\frac{t}{2}}^t g(0, s) \dx{s} &\lesssim t \sup_{0 \leq s \leq t} \norm*{\<1>_{0, s}}_{-\al} + \sup_{0 \leq s \leq t} \norm*{\<1>_{0, s}}^{\frac{2}{1+\al}}_{-\al} \int_{\frac{t}{2}}^t \norm*{\nabla v_{0, s}}^{\frac{2\al}{1+\al}}_{L^\infty_x} \dx{s} + t \sup_{0 \leq s \leq t} \norm*{\<1>_{0, s}}^{\frac{2}{1-\al}}_{-\al} \\
        & \quad + t \sup_{0 \leq s \leq t} \norm*{\<2>_{0, s}}^{\frac{2}{2-\al}}_{-\al} + t^{1-\beta}
    \end{align}
    and for any $p \geq 1$ we can estimate
    \begin{align}
        \int_{\frac{t}{2}}^t \E{\norm*{\nabla v_{0, s}}^{\frac{2\al p}{1+\al}}_{L^\infty_x}}^{\frac{1}{p}} \dx{s} \stackrel{\eqref{gradvest}}{\lesssim} \int_{\frac{t}{2}}^t s^{-\frac{(1 + \beta) 2\al }{1 + \al}}\dx{s} \lesssim t^{1-\frac{(1 + \beta) 2\al }{1 + \al}}.
    \end{align}
    Moreover, by Proposition~\ref{stochest} for every $p<\infty$ there exists $r>0$ such that 
    \begin{align}
        & \E{\sup_{0 \leq s \leq t} \norm*{\<1>_{0, s}}^p_{-\al}}^{\frac{1}{p}} \lesssim (1+t)^r, \label{estsup1}
        \\
        & \E{\sup_{0 \leq s \leq t}\norm*{\<2>_{0, s}}^p_{-\al}}^{\frac{1}{p}} \lesssim (1+t)^r. \label{estsup2}
    \end{align}
    Any positive power of $t$ can brutally be bounded by $C_{\sigma} e^{\sigma t}$ for $\sigma > 0$, thus we have for any $p \geq 1$
    \begin{align}
        \E{\bra*{\int_{\frac{t}{2}}^t g(0, s) \dx{s}}^p}^{\frac{1}{p}} \lesssim \ee^{\sigma t} 
    \end{align}
    where we implicitly used H\"older's inequality in expectation.
    
    Also, again for any $p \geq 1$ we have
    \begin{align}
        \bra*{\int_{\frac{t}{2}}^t \bra*{\bra*{t-s} \wedge 1}^{-\kappa} \E{\norm*{v_{0, s}}^{p}_{L^\infty_x}}^{\frac{2}{p}}\dx{s}}^{\frac{1}{2}} \stackrel{\eqref{cvest}}{\lesssim} \bra*{\int_{\frac{t}{2}}^t \bra*{\bra*{t-s} \wedge 1}^{-\kappa} s^{-1-\al} \dx{s}}^{\frac{1}{2}} \lesssim \bra*{t \wedge 1}^{-\frac{\kappa + \al}{2}}
    \end{align}
    and similarly
    \begin{align}
        \bra*{\int_{\frac{t}{2}}^t \bra*{\bra*{t-s} \wedge 1}^{-\kappa - \al}  \E{\norm{v_{0, s}}^{p}_{2\al}}^{\frac{2}{p}} \dx{s}}^{\frac{1}{2}} &\stackrel{\eqref{cvest}}{\lesssim} \bra*{t \wedge 1}^{-\frac{\kappa + 5 \al}{2}}.
    \end{align} 
    Dividing by $\norm*{h}_{L^2_x}$, taking the supremum and using \eqref{estsup1}, \eqref{estsup2} and \eqref{expmoments}, we conclude that
    \begin{align}
        \E{\norm*{J_{0, t}}^{p}_{L^2_x \to H^{\kappa}_x}}^{\frac{1}{p}} \lesssim_{\alpha,\kappa, L} 
         (t\wedge1)^{- \frac{\kappa + 5\al}{2}} \ee^{-\bra*{m - 10\sigma - \frac{2\ln2}{\theta}}t},
    \end{align}
    where we have used H\"older's inequality in probability repeatedly.
\end{proof}

\begin{appendices}\label{sec:app}

\addtocontents{toc}{\protect\setcounter{tocdepth}{0}}

\section{Estimate on the renormalization constant}

\begin{prop}\label{renormc}
   The following estimate holds for any $\beta\in(0,1)$ and $t>0$,
   \begin{equs}
    c_{t,\infty} = 2 \int_t^\infty \dd s \,H_{2s}(0) \lesssim_\beta t^{-\frac{\beta}{2}}.    
   \end{equs}
\end{prop}

\begin{proof}
    By a simple computation in Fourier space we have that 
    \begin{align}
        c_{t, \infty} = \sum_{k \in \Z^2} \frac{\ee^{-t (m + \abs*{k}^2)}}{m + \abs*{k}^2}.
    \end{align}
    Noticing that  $\ee^{-t (m + \abs*{k}^2)} \lesssim_\gamma \frac{t^{-\frac{\beta}{2}}}{(m + \abs*{k}^2)^{\frac{\beta}{2}}}$ for any $\beta\in(0,1)$ we get the assertion since the sum 
    $\sum_{k \neq 0} \frac{1}{\abs*{k}^{2 + \beta}}$ is finite.
\end{proof}

\section{Besov-norm estimates}

\begin{lem}[{\cite[p. 308, (A.2)]{TW20}}] \label{besovmono}
Let $\al \leq \beta$ and $p, q \geq 1$, then we have
   \begin{align}
       \norm*{f}_{\mathcal{B}_{p, q}^{\al}} \leq \norm*{f}_{\mathcal{B}_{p, q}^{\beta}}.
   \end{align}
\end{lem}
\begin{lem}[{\cite[p. 309, Proposition A.5]{TW20}}]\label{besovsemi}
    Let $\al \leq \beta $ and $p, q \geq 1$, then it holds that
    \begin{align}
        \norm*{S_t f}_{\mathcal{B}_{p, q}^{\beta}} \lesssim \ee^{-mt} \bra*{t \wedge 1}^{\frac{\al - \beta}{2}} \norm*{f}_{\mathcal{B}_{p, q}^{\al}}
    \end{align}
    where $S_t$ denotes the semigroup generated by $\Delta - m$ for $m \geq 0$.
\end{lem}

\begin{lem}[{\cite[p. 309, Proposition A.6]{TW18}}]\label{besovmult1}
    Let $\al \geq 0$ and $p, q \geq 1$, then 
    \begin{align}
        \norm*{fg}_{\mathcal{B}^{\al}_{p,q}} \lesssim \norm*{f}_{\mathcal{B}^{\al}_{p_1, q_1}} \norm*{g}_{\mathcal{B}^{\al}_{p_2, q_2}}
    \end{align}
    where $\frac{1}{p} = \frac{1}{p_1} + \frac{1}{p_2}$as well as $\frac{1}{q} = \frac{1}{q_1} + \frac{1}{q_2}$.
\end{lem}

\begin{lem}[{\cite[p. 309, Proposition A.7]{TW18}}]\label{besovmult2}
    Let $\al < 0$ and $\beta > 0$ such that $\al + \beta > 0$ and $p, q \geq 1$, then 
    \begin{align}
        \norm*{fg}_{\mathcal{B}^{\al}_{p,q}} \lesssim \norm*{f}_{\mathcal{B}^{\al}_{p_1, q_1}} \norm*{g}_{\mathcal{B}^{\beta}_{p_2, q_2}}
    \end{align}
    where $\frac{1}{p} = \frac{1}{p_1} + \frac{1}{p_2}$ and $\frac{1}{q} = \frac{1}{q_1} + \frac{1}{q_2}$.
\end{lem}

\section{Stochastic estimates}\label{sec:stochest}

In this section we provide an alternative argument for the stochastic estimates in \cref{stochest} using the spectral gap inequality \eqref{spg} for the noise $\xi$ in the spirit of \cite[Section 5]{IORT20} 
and \cite{LOTT21}. 

Let $F$ be cylindrical in $\xi$, i.e. there is $n \in \N$, $\bar{F} \in C^{\infty}_c(\R^n, \R)$ and $h_1, \dots, h_n \in L^2_{t, x}$ such that $F(\xi) = \bar{F}(\xi(h_1), \dots, \xi(h_n))$. Since $\xi$ is Gaussian it satisfies the following spectral gap inequality (cf. \cite[p. 652, Proposition 4.1]{DO20})
\begin{align}\label{spg}
    \E{\abs*{F(\xi) - \E{F(\xi)} }^2} \leq \E{\norm*{ \dxi F(\xi)}^2_{L^2_{t, x}} }
\end{align}
where $\dxi$ denotes the Malliavin derivative with respect to the noise $\xi$. This in turn can be used to construct the singular products as follows.

\begin{proof}[Proof of \cref{stochest}]
    For simplicity we assume that the noise $\xi$ is smooth. By the spectral gap inequality \eqref{spg} we know that for nice enough functionals $\Pi\pra*{\xi}$ and $p \geq 2$ there holds
    \begin{align}
        \Exp^{\frac{1}{p}} |\Pi\pra*{\xi} - \Exp \Pi\pra*{\xi}|^p \lesssim_p \Exp^{\frac{1}{p}} \norm*{\dxi \Pi\pra*{\xi}}^p_{L^2_{t,x}}.
    \end{align}
    By duality, an estimate of the form 
    \begin{align}
        \abs*{ \Exp \dxi\Pi\pra*{\xi}(\delta \xi) } \leq C \, \Exp^{\frac{1}{q}} \norm*{\delta \xi}^{q}_{L^2_{t, x}} \label{eq:dual_est}
    \end{align}
    for any $\delta \xi : \Omega \to L^2_{t,x}$\footnote{where $\Omega$ denotes the underline probability space}, where $q\in[1,2]$ is the dual exponent of $p$, implies 
    \begin{align}
        \Exp^{\frac{1}{p}}\norm*{\dxi \Pi\pra*{\xi}}^p_{L^2_{t,x}} \leq C.
    \end{align}
    For $t > 0$ and $x \in \T^2$ we consider $\Pi_t(x) \in \set*{\<1>_{0, t}(x), \<2>_{0,t}(x), \<3>_{0, t}(x)}$, where
    \begin{equs}
     \<2>_{0,t}(x) := \<1>^2_{0, t}(x) - c_{0, t}, \quad \<3>_{0, t}(x) := \<1>^3_{0, t}(x) - 3c_{0, t} \<1>_{0, t}(x), 
    \end{equs}
    for $c_{0,t} = \Exp \<1>_{0,t}(0)^2$. We treat $\Pi_t(x) \equiv \Pi_t[z](x)$ as a functional of $\xi$ and aim to prove the following stochastic estimates (replacing $x$ by $0$ using stationarity) which are uniform in $m$,
    \begin{align}
        &\Exp^{\frac{1}{p}}  \abs*{\Pi_{t \lambda}(0)}^p \lesssim \lambda^{-|\Pi|\al} \sqrt{t}^{|\Pi|\al}, \label{eq:space_reg_ann}
        \\
        &\Exp^{\frac{1}{p}} \abs*{ \bra*{\Pi_{t+r} - \Pi_{t}}_{\lambda}(0)}^p 
        \lesssim \lambda^{-(|\Pi|+1)\alpha} \sqrt{r}^{\alpha} \sqrt{t+r}^{|\Pi|\alpha} \label{eq:time_reg_ann}
    \end{align}
    for every $\alpha\in(0,\frac{1}{|\Pi|+1})$, where $|\Pi|=1,2,3$ for $\Pi=\<1>, \<2>, \<3>$ respectively and $(\cdot)_\lambda$ denotes convolution with a suitable semigroup $\psi_\lambda$. By a Kolmogorov-type continuity criterion, see for \cite[Lemma 10]{MW17I}, 
    we then obtain \eqref{eq:stochest}. It is important to stress the uniformity of our estimates in $m$ which allows us to ensure that $m_*$ in Theorem~\ref{thm:weak_grad_est} does not depend 
    on $m$\footnote{or equivalently $\theta$ in \cref{stochest} does not depend on $m$}. This will be obvious in what follows except \eqref{eq:uniform_m} where one should pay attention on how the power on $\sqrt{r}$ 
    is chosen.  
    
    \medskip
    
    For $\delta \xi\in L^q_\omega L^2_{t,x}$ we let $\delta \<1>_{0,t}(x) := \dxi \<1>_{0,t}(\delta \xi) = \int_0^t \dx{s} H_{t-s} * \delta \xi(s, x)$ and consider $\delta \Pi_t \in \{\delta \<1>_{0,t}, \delta \<1>_{0,t} \<1>_{0,t}, \delta \<1>_{0,t} \<2>_{0,t}\}$. 
    As in \cite{LOTT21}, in order to prove \eqref{eq:space_reg_ann} and \eqref{eq:time_reg_ann} we appeal to duality and derive the following estimates for the Malliavin derivative of $\Pi_t$, 
    \begin{align}
        &\Exp^{\frac{1}{q'}} \abs*{\delta \Pi_{t \lambda}(0)}^{q'} \lesssim \lambda^{-|\Pi|\al} \sqrt{t}^{|\Pi|\al} \norm*{\Exp^{\frac{1}{q}} \abs*{\delta \xi}^{q}}_{L^2_{t,x}}, \label{eq:space_reg_ann_mal}
         \\
        &\Exp^{\frac{1}{q'}} \abs*{\bra*{\delta \Pi_{t + r} - \delta \Pi_t}_{\lambda}(0)}^{q'} \lesssim  
        \lambda^{-(|\Pi|+1)\alpha} \sqrt{r}^{\alpha} \sqrt{t+r}^{|\Pi|\alpha}
        \norm*{\Exp^{\frac{1}{q}} {\abs*{\delta \xi}^{q}}}_{L^2_{t,x}}, \label{eq:time_reg_ann_mal}
    \end{align}
    for all $q' < q < 2$. Note that in \eqref{eq:space_reg_ann_mal} and \eqref{eq:time_reg_ann_mal} we ask for an estimate of the 
    $L^{q'}_\omega$-norm by the $L^2_{t,x}L^{q}_{\omega}$-norm which is stronger than the $L^{q}_{\omega}L^2_{t,x}$-norm for 
    $q<2$, therefore implying the dual estimate \eqref{eq:dual_est}. As in \cite{LOTT21} estimating the $L^{q'}_\omega$-norm
    for all $q'<q<2$ allows us to proceed inductively, namely, in order to derive the dual estimate for $\delta \<2>_{0,t}$ we need
    the stronger estimate on $\delta \<1>_{0,t}$ and similarly for $\delta \<3>_{0,t}$.
    
    \medskip
    
    To this end, we denote by $\bar w$ the $L^2_{t,x}L^{q}_{\omega}$-norm on the r.h.s. of  \eqref{eq:space_reg_ann_mal} and \eqref{eq:time_reg_ann_mal} and introduce another scaling parameter $\Lambda$, coming from 
    $(\cdot)_\Lambda$. We estimate commutators of the form 
    \begin{equs}
     ([\delta \Pi, (\cdot)_\lambda]\Pi_\lambda)_{\Lambda}(0) = 
     \int \dx{x} \, \psi_{\Lambda}(-x) \int \dx{y} \, \psi_{\lambda}(y) \bra*{\delta \Pi(x-y) - \delta \Pi(x)} \Pi_\lambda(x-y).
    \end{equs}
    Using the Cauchy--Schwarz inequality in the $x$-variable we have\footnote{Here $p\geq 2$ satisfies $\frac{1}{q'} = \frac{1}{q} +\frac{1}{p}$.}
    \begin{equs}
     {} &\Exp^{\frac{1}{q'}} \abs*{\bra*{\pra*{\delta \<1>_{0,t}, \bra*{\cdot}_{\lambda}} \Pi_{\lambda}}_{\Lambda}(0)}^{q'} 
     \\
     & \quad = \Exp^{\frac{1}{q'}} \abs*{\int \dx{x} \psi_{\Lambda}(-x) \int \dx{y} \psi_{\lambda}(y) \bra*{\delta \Pi(x-y) - \delta \Pi(x)} \Pi_{\lambda}(x-y)}^{q'}
     \\
     & \quad \leq \int \dx{x} \abs*{\psi_{\lambda}(y)} \norm*{\psi_{\Lambda}}_{L^2} \norm*{\Exp^{\frac{1}{q}} \abs*{\delta \Pi(x-y) - \delta \Pi(x)}^{q}}_{L^2_x} \Exp^{\frac{1}{p}} \abs*{\Pi_{\lambda}(0)}^p. \label{comest}
    \end{equs}

   \medskip
    
    For \eqref{eq:space_reg_ann_mal} we let $\delta\Pi= \delta\<1>_t$ and $\Pi\in \{\<1>_{0,t}, \<2>_{0,t}\}$. Using the interpolation inequality \cref{interpolationlemma} and the Cauchy-Schwarz inequality in the $s$-variable
    we see that 
    \begin{align}
        \quad \quad \norm*{\Exp^{\frac{1}{q}} \abs*{\delta \<1>_{0,t}(x - y) - \delta \<1>_{0,t}(x)}^{q}}_{L^2_x} &\leq \int_0^t \dx{s} \int \dx{z} \, \abs*{H_{t-s}(z-y) - H_{t-s}(z)} \norm*{\Exp^{\frac{1}{q}} \abs*{\delta \xi}^{q}}_{L^2_x} \\
        &\leq \abs*{y}^{1-\alpha} \bra*{\int_0^t \dx{s} \, \ee^{-2m(t-s)} \bra*{t-s}^{-1+\alpha}}^{\frac{1}{2}} \bar w \\
        &\leq \abs*{y}^{1-\alpha} \sqrt{t}^{\alpha} \bar w, \label{diffdelta}
    \end{align}
    for all $\alpha\in(0,1)$ uniformly in $m$. Combining \eqref{comest} and \eqref{diffdelta} yields
    \begin{align}\label{comest2}
        \Exp^{\frac{1}{q'}} \abs*{\bra*{\pra*{\delta \<1>_{0,t}, \bra*{\cdot}_{\lambda}} \Pi_{\lambda}}_{\Lambda}(0)}^{q'} \lesssim \Lambda^{-1} \lambda^{1-\alpha}  \sqrt{t}^{\alpha} 
        \Exp^{\frac{1}{p}} \abs*{\Pi_{\lambda}(0)}^p \bar w.
    \end{align}
    Using \eqref{eq:space_reg_ann} and the dyadic summation identity
    \begin{equs}
     \big([\delta\Pi, (\cdot)_\lambda]\Pi\big)_\Lambda = \sum_{\substack{k\geq1 \\ \lambda'=\frac{\lambda}{2^k}}}\big([\delta\Pi, (\cdot)_{\lambda'}](\Pi)_{\lambda'}]\big)_{\Lambda+\lambda-2\lambda'},
    \end{equs}
    we obtain via \eqref{comest2}
    \begin{equs}
      \Exp^{\frac{1}{q'}} \abs*{\bra*{\pra*{\delta \<1>_{0,t}, \bra*{\cdot}_{\lambda}} \Pi}_{\Lambda}(0)}^{q'} \lesssim \Lambda^{-1} \lambda^{1-|\Pi|\alpha} \sqrt{t}^{|\Pi|\alpha} \bar w. 
    \end{equs}
    A simple post-processing of the last estimate choosing $\Lambda \sim \lambda$ gives 
    \begin{equs}
     \Exp^{\frac{1}{q'}} |(\delta \<1>_ {0,t}\Pi)_\lambda(0)|^{q'} \lesssim \lambda^{-|\Pi|\alpha} \sqrt{t}^{|\Pi|\alpha} \bar w, 
     \label{eq:malder}
    \end{equs}
    therefore yielding \eqref{eq:space_reg_ann_mal}. 
    
    \medskip
    
     For \eqref{eq:time_reg_ann_mal} we write $\delta \<1>_{0,t + r} \Pi_{t+r} - \delta\<1>_{0,t} \Pi_{t} = \delta\<1>_{0,t+r}\bra*{\Pi_{t+r} - \Pi_{t}} + \Pi_{t}\bra*{\delta\<1>_{0,t+r} - \delta\<1>_{0,t}}$ and use \eqref{comest} 
     for the pairs $\delta\Pi =  \delta\<1>_{0,t+r}$, $\Pi = \Pi_{t+r} - \Pi_t$ and  $\delta\Pi =\delta\<1>_{0, t+r} - \delta\<1>_{t}$, $\Pi = \Pi_{t}$. For the first pair we apply \eqref{comest2} to get
        \begin{align}\label{diffpi}
        &\Exp^{\frac{1}{q'}} \abs*{\bra*{\pra*{\delta \<1>_{0,t+r}, \bra*{\cdot}_{\lambda}} \bra*{\Pi_{t+r} - \Pi_t}_{\lambda}}_{\Lambda}(0)}^{q'} 
        \lesssim  \Lambda^{-1} \lambda^{1-\alpha} \sqrt{t +r}^{\alpha} \Exp^{\frac{1}{p}} \abs*{\bra*{\Pi_{t+r} - \Pi_{t}}_{\lambda}(0)}^p \bar w.
    \end{align}
    Plugging in \eqref{eq:time_reg_ann} for $\Pi_t\in\{\<1>_{0,t},\<2>_{0,t}\}$ and proceeding as for \eqref{eq:malder} yields 
    \begin{equs}\label{eq:time_cont1}
     \Exp^{\frac{1}{q'}} \abs*{\bra*{\delta \<1>_{0,t+r} \bra*{\Pi_{t+r} - \Pi_t}}_{\lambda}(0)}^{q'} \lesssim 
     \lambda^{-(|\Pi|+2)\alpha} \sqrt{r}^\alpha \sqrt{t +r}^{(|\Pi|+1) \alpha} \bar w.
    \end{equs}
     For the second pair, abbreviating $\delta\Pi_{t,t+r} := \delta \<1>_{0,t+r} - \delta \<1>_{0,t}$, \eqref{comest} implies
    \begin{align}
        &\Exp^{\frac{1}{q'}} \abs*{\bra*{\pra*{\delta \<1>_{0,t+r} - \delta \<1>_{0,t}, \bra*{\cdot}_{\lambda}} \Pi_{t \lambda}}_{\Lambda}(0)}^{q'} 
        \\
        & \quad \leq \int \dx{x} \abs*{\psi_{\lambda}(y)} \norm*{\psi_{\Lambda}}_{L^2} \norm*{\Exp^{\frac{1}{q}} \abs*{\delta\Pi_{t,t+r}(x - y) - \delta\Pi_{t,t+r}(x)}^{q}}_{L^2_x} \Exp^{\frac{1}{p}} \abs*{\Pi_{t\lambda}(0)}^p.
        \label{comestdif}
    \end{align}
    We use the following estimate
    \begin{align}
        \norm*{\Exp^{\frac{1}{q}} \abs*{\delta\Pi_{t,t+r}(x - y) - \delta\Pi_{t,t+r}(x)}^{q}}_{L^2_x} 
        \lesssim \abs*{y}^{1-2\alpha} \sqrt{r}^{\alpha} \sqrt{t+r}^{\alpha} \bar w \label{difff}
    \end{align}
    for every $\alpha\in(0,\frac{1}{2})$, which itself is an interpolation\footnote{using $\beta=\alpha$ and $\frac{\alpha}{1-\alpha} +\frac{1-2\alpha}{1-\alpha}$} of the two estimates
    \begin{align}
       & \norm*{\Exp^{\frac{1}{q}} \abs*{\delta\Pi_{t, t+r}(x - y) - \delta\Pi_{t, t+r}(x)}^{q}}_{L^2_x} \lesssim \abs{y}^{1-\beta}  \sqrt{t+r}^{\beta} \bar w, \label{eq:int1}
        \\
        &\norm*{\Exp^{\frac{1}{q}} \abs*{\delta\Pi_{t, t+r}(x - y) - \delta\Pi_{t, t+r}(x)}^{q}}_{L^2_x} \lesssim \sqrt{r}^{1-\beta} \sqrt{t+r}^{\beta} \bar w, \label{eq:int2}
    \end{align}
    for every $\beta\in(0,1)$. 
    Estimate \eqref{eq:int1} follows along the same lines as \eqref{diffdelta} using the triangle inequality. For \eqref{eq:int2} using again the triangle inequality, translation invariance and the semigroup property in the form
    \begin{equs}
     \delta\<1>_{0,t+r}(x) = \int \dd z \, \ee^{-mr} \tilde H_r(z) \delta\<1>_{0,t}(x-z) + \underbrace{\int_t^{t+r} \dd s \, H_{t-s}*\delta\xi(s,x)}_{=:\delta\<1>_{t,t+r}},
    \end{equs}
    where $\tilde H_r$ stands for the massless heat kernel, we observe
    \begin{align}
        \norm*{\Exp^{\frac{1}{q}} \abs*{\delta\Pi_{t, t+r}(x - y) - \delta\Pi_{t, t+r}(x)}^{q}}_{L^2_x} 
        & \leq 2 \norm*{\Exp^{\frac{1}{q}} \abs*{\delta \<1>_{0,t+r} - \delta \<1>_{0,t}}^{q}}_{L^2_x}
        \\
        & \lesssim \int \dx{z} \, H_r(z) \norm*{\Exp^{\frac{1}{q}} \abs*{\delta \<1>_{0,t}(\cdot-z) - \delta \<1>_{0,t}}^{q}}_{L^2_x}
        \\ & \quad  + \abs{e^{-mr} - 1} \norm*{\Exp^{\frac{1}{q}} \abs*{\delta \<1>_{0,t}}^{q}}_{L^2_x} 
        + \norm*{\Exp^{\frac{1}{q}} \abs*{\delta \<1>_{t, t+r}}^{q}}_{L^2_x} 
        \\
        & =: I_1 + I_2 + I_3.
    \end{align}
    For $I_1$ using \eqref{diffdelta} we obtain
    \begin{equs}
     \int \dx{z} \, H_r(z) \norm*{\Exp^{\frac{1}{q}} \abs*{\delta \<1>_{0,t}(\cdot-z) - \delta \<1>_{0,t}}^{q}}_{L^2_x} \lesssim \sqrt{r}^{1-\beta} \sqrt{t}^{\beta}  \bar w
     \lesssim \sqrt{r}^{1-\beta} \sqrt{t+r}^{\beta}  \bar w.
    \end{equs}
    To estimate $I_2$ we use Young's inequality for convolution, the Cauchy--Schwarz inequality in the $s$-variable and the H\"older's inequality again in the $s$-variable to treat the
    integral of the exponential yielding 
    \begin{equs}
     \norm*{\Exp^{\frac{1}{q}} \abs*{\delta \<1>_{0,t}}^{q}}_{L^2_x} & \leq \int _0^t \dd s \, \|H_{t-s}\|_{L^1_x} \norm*{\Exp^{\frac{1}{q}} \abs*{\delta\xi}^q}_{L^2_x} 
     \\
     & \leq \left(\int_0^t \dd s \, \ee^{-2m(t-s)} \right)^{\frac{1}{2}} \left(\int _0^t \dd s \,  \norm*{\Exp^{\frac{1}{q}} \abs*{\delta\xi}^q}_{L^2_x}^2\right)^{\frac{1}{2}}
     \lesssim \frac{1}{\sqrt{m}^{1-\beta}} \sqrt{t}^{\beta} \bar w \label{eq:deltalol_est} 
    \end{equs}
    for every $\beta\in[0,1)$. This in turn implies the estimate 
    \begin{equs}
     \abs{e^{-mr} - 1} \norm*{\Exp^{\frac{1}{q}} \abs*{\delta \<1>_{0,t}}^{q}}_{L^2_x} \lesssim \abs{e^{-mr} - 1} \frac{1}{\sqrt{m}^{1-\beta}} \sqrt{t}^{\beta} \bar w
     \lesssim \sqrt{r}^{1-\beta} \sqrt{t+r}^{\beta} \bar w, \label{eq:uniform_m}
    \end{equs}
    where the implicit constant is uniform in $m$. To estimate $I_3$ we use \eqref{eq:deltalol_est} for $\beta=0$ and a change of variables in $s$ which leads to
    \begin{equs}
     \norm*{\Exp^{\frac{1}{q}} \abs*{\delta \<1>_{t, t+r}}^{q}}_{L^2_x} \lesssim \sqrt{r} \, \bar w 
     \lesssim \sqrt{r}^{1-\beta} \sqrt{t+r}^{\beta} \bar w. 
    \end{equs}
    In total, \eqref{comestdif} and \eqref{difff} imply the estimate
    \begin{equs}
     \Exp^{\frac{1}{q'}} \abs*{\bra*{\pra*{\delta \<1>_{0,t+r} - \delta \<1>_{0,t}, \bra*{\cdot}_{\lambda}} \Pi_{t \lambda}}_{\Lambda}(0)}^{q'} 
     \lesssim \Lambda^{-1} \lambda^{1-2\alpha} \sqrt{r}^{\alpha} \sqrt{t+r}^{\alpha} \Exp^{\frac{1}{p}} \abs*{\Pi_{t\lambda}(0)}^p \bar w.
    \end{equs}
    Plugging in \eqref{eq:space_reg_ann} for $\Pi_t\in \{\<1>_{0,t}, \<2>_{0,t}\}$ and proceeding as in \eqref{eq:time_cont1} gives 
    \begin{equs}
     \Exp^{\frac{1}{q'}} \abs*{\bra*{(\delta \<1>_{0,t+r} - \delta \<1>_{0,t}) \Pi_{t}}_{\lambda}(0)}^{q'}
     \lesssim \lambda^{-(|\Pi|+2)\alpha} \sqrt{r}^\alpha \sqrt{t +r}^{(|\Pi|+1) \alpha} \bar w. \label{eq:time_cont2}
    \end{equs}
    Combining \eqref{eq:time_cont1} and \eqref{eq:time_cont2} implies \eqref{eq:time_reg_ann_mal}.
\end{proof}

\begin{lem}\label{interpolationlemma}
    For all $\alpha \in (0, 1)$ the following estimate holds
    \begin{align}
        \int \dx{z} \, \abs*{H_{t-s}(z-y) - H_{t-s}(z)}
           \lesssim \ee^{-m(t-s)}  \abs*{y}^{\alpha} \sqrt{t-s}^{-\alpha}.
 \end{align}
\end{lem}
\begin{proof}
Interpolating the two estimates
\begin{align}
    \int \dx{z} \, \abs*{H_{t-s}(z-y) - H_{t-s}(z)} \leq 2 \norm*{H_{t-s}}_{L^1_x} = 2\ee^{-m(t-s)}
\end{align}
and \begin{align}
    \int \dx{z} \, \abs*{H_{t-s}(z-y) - H_{t-s}(z)} \leq \norm*{\nabla H_{t-s}}_{L^1_x} \abs*{y} \leq \ee^{-m(t-s)} \sqrt{t-s}^{-1} \abs*{y}
\end{align}
yields the assertion.
\end{proof}

\section{Estimates on the remainder}\label{sec:remv}

\begin{lem}\label{lpvest}
    Let $\alpha>0$ be sufficiently small. For every $p <\infty$
    \begin{align}
        \sup_{t\leq 1} t^{\frac{1}{2}} \norm*{v_{0, t}}_{L^p_x} \leq C,
    \end{align}
    where $C$ depends polynomially on $\sup_{t\leq 1} \|\<k>_{0,t}\|_{C^{-\alpha}}$ for $k=1,2,3$ and is uniform in the initial condition $f$. In particular, $C$ has finite moments of every order.
\end{lem}

\begin{proof}
    Follows from \cite[Proposition~3.7]{TW18}. The constant $c_{t,\infty}$ in \cref{renormc} can be absorbed into the terms $\<2>_{0,t}$ and $\<3>_{0,t}$ which together with Proposition~\ref{renormc}
    yield
    \begin{equs}
     \sup_{t\leq 1} t^{\alpha'} \|\<2>_{0,t} - c_{t,\infty}\|_{-\alpha} \lesssim \sup_{t\leq 1} \|\<2>_{0,t}\|_{C^{-\alpha}}, \  \sup_{t\leq 1} t^{\alpha'} \|\<3>_{0,t} - 3c_{t,\infty} \<1>_{0,t}\|_{-\alpha} \lesssim \max_{k=1,3}\sup_{t\leq 1} \|\<k>_{0,t}\|_{C^{-\alpha}},
    \end{equs}
    for any $\alpha'>0$, allowing us to apply \cite[Proposition~3.7]{TW18}. 
\end{proof}

\begin{lem}\label{cvest} 
    Let $\alpha>0$ be sufficiently small. Then for every $\kappa>0$ sufficiently small the following estimate holds 
    \begin{align}
        \sup_{t\leq 1} t^{\frac{1}{2}+\kappa} \norm*{v_{0, t}}_{\kappa} \leq C, 
    \end{align}
    where $C\equiv$ depends polynomially on $\sup_{t\leq 1} \|\<k>_{0,t}\|_{-\alpha}$ for $k=1,2,3$ and is uniform in the initial condition $f$. 
\end{lem}

\begin{proof}
   The statement follows essentially from the proof of Lemma 5.1 for $s=\frac{t}{2}$ in \cite{TW20} working with  $\<1>_{0,t}$, $\<2>_{0,t} - c_{t,\infty}$, $\<3>_{0,t} - 3 c_{t,\infty}\<1>_{0,t}$ as we explained in the proof of
   \cref{lpvest}. The terms $I_6$ and $I_7$ in the notation of \cite[proof of Lemma~5.1]{TW20} can be ignored. 
\end{proof}

\begin{lem}\label{gradvest}
    Let $\alpha>0$ be sufficiently small. Then for any $\eps > 0$ the following estimate holds
    \begin{align}
        \sup_{t \leq 1} t^{1 + \eps} \norm*{\nabla v_{0, t}}_{L^\infty_x} \leq C,
    \end{align}
    where $C$ depends polynomially on $\sup_{t\leq 1} \|\<k>_{0,t}\|_{-\alpha}$ for $k=1,2,3$ and is uniform in the initial condition $f$. 
\end{lem}

\begin{proof}
    To ease the notation we set $\eta:=\max_{k=1,2,3}\sup_{t\leq 1} \|\<k>_{0,t}\|_{-\alpha}$. By Duhamel's formula, we have 
    \begin{align}
        \norm*{\nabla v_{0, t}}_{L^\infty_x} & \leq \norm*{\nabla H_{\frac{t}{2}} * v_{0, \frac{t}{2}}}_{L^\infty_x} + \sum_{k = 0}^3 \int_{\frac{t}{2}}^t \norm*{\nabla H_{t-r} * \bra*{ \<k>_{0, r} v_{0, r}^{3-k}}}_{L^\infty_x} \dx{r} \\
       & \quad + 3 \int_{\frac{t}{2}}^t c_{r, \infty} \norm*{\nabla H_{t-r} * \bra*{\<1>_{0,r} + v_{0, r}}}_{L^{\infty}_x} \dx{r}.
    \end{align}
    Note in the following that $B^{\eps}_{\infty, \infty}(\T^2) = C^{\eps} \hookrightarrow  L^{\infty}_x $ continuously for any $\eps > 0$.
    Then, first of all, by Young's inequality and \eqref{lpvest}, we have
    \begin{align}
        \norm*{\nabla H_{\frac{t}{2}} * v_{0, \frac{t}{2}}}_{L^\infty_x} &\leq \norm*{\nabla H_{\frac{t}{2}}}_{L^{p'}_x} \norm*{v_{0, \frac{t}{2}}}_{L^p_x} \lesssim t^{-\frac{1}{2} - \frac{1}{p}} t^{-\frac{1}{2}} 
        = t^{-1-\eps}
    \end{align}
    for $p$ large enough.
    In the same vain, using $\eqref{lpvest}$ and $p$ large enough yields
    \begin{align}
        \int_{\frac{t}{2}}^t \norm*{\nabla H_{t-r} * v^3_{0, r}}_{L^\infty_x} \dx{r} &\leq \int_{\frac{t}{2}}^t \norm*{\nabla H_{t-r}}_{L^{p'}_x} \norm*{v^3_{0, r}}_{L^p_x}  \dx{r} = \int_{\frac{t}{2}}^t \norm*{\nabla H_{t-r}}_{L^{p'}_x} \norm*{v_{0, r}}^3_{L^{3p}_x}  \dx{r} \\
        & \leq \int_{\frac{t}{2}}^t (t-r)^{-\frac{1}{2} - \frac{1}{p}} r^{-\frac{3}{2}} \dx{r} \lesssim t^{-1-\eps}.
    \end{align}
    Moreover, using the semigroup property of the heat kernel and Young's inequality again we note
    \begin{align}
        \int_{\frac{t}{2}}^t \norm*{\nabla H_{t-r} * \<3>_{0, r}}_{L^\infty_x} \dx{r} &= \int_{\frac{t}{2}}^t \norm*{\nabla H_{\frac{t-r}{2}} * H_{\frac{t-r}{2}} * \<3>_{0, r}}_{L^\infty_x} \dx{r} \\
        &\leq \int_{\frac{t}{2}}^t \norm*{\nabla H_{\frac{t-r}{2}}}_{L^1_x} \norm*{H_{\frac{t-r}{2}} * \<3>_{0, r}}_{L^\infty_x} \dx{r}.
    \end{align}
    Moreover, by \cref{besovsemi} we conclude
    \begin{align}
    \int_{\frac{t}{2}}^t \norm*{\nabla H_{t-r} * \<3>_{0, r}}_{L^\infty_x} \dx{r} &\lesssim \int_{\frac{t}{2}}^t (t-r)^{-\frac{1}{2}} (t-r)^{-\al - \eps}\dx{r} \lesssim \eta t^{\frac{1}{2} - \al -\eps}.
    \end{align}
    Similarly, using \eqref{cvest} and \cref{besovsemi} we end up with
    \begin{align}
        \int_{\frac{t}{2}}^t  \norm*{\nabla H_{t-r} * \bra*{ v^2_{0, r} \<1>_{0, r}}}_{L^\infty_x} \dx{r} &\leq \int_{\frac{t}{2}}^t  (t-r)^{-\frac{1}{2} - \al - \eps} \norm*{v^2_{0, r} \<1>_{0, r}}_{-\al} \dx{r} \\
        &\leq \int_{\frac{t}{2}}^t  (t-r)^{-\frac{1}{2} - \al - \eps} \norm*{v^2_{0, r}}_{2\al} \norm*{\<1>_{0, r}}_{-\al} \dx{r} \\
        &\lesssim \eta \int_{\frac{t}{2}}^t (t-r)^{-\frac{1}{2} - \al - \eps} r^{-1-2\al} \dx{r} \lesssim \eta t^{-\frac{1}{2} - 3\al - \eps}
    \end{align}
    and in the same vain
    \begin{align}
         \int_{\frac{t}{2}}^t \norm*{\nabla H_{t-r} * \bra*{ v_{0 , r} \<2>_{0, r}}}_{L^\infty_x}  \dx{r} &\lesssim  \int_{\frac{t}{2}}^t (t-r)^{-\frac{1}{2} - \al - \eps} \norm*{v_{0, r} \<2>_{0, r}}_{-\al} \dx{r} \\
         &\lesssim \eta \int_{\frac{t}{2}}^t (t-r)^{-\frac{1}{2} - \al - \eps} r^{-\frac{1}{2} - \al} \dx{r} \lesssim \eta t^{-2\al-\eps}.
    \end{align}
    Finally, using \cref{besovsemi}, \eqref{cvest} and \eqref{renormc} we get
    \begin{align}
        \int_{\frac{t}{2}}^t c_{r, \infty} \norm*{\<1>_{0, r} + v_{0, r}}_{L^{\infty}_x} \dx{r} &\lesssim \int_{\frac{t}{2}}^t \bra*{t-r}^{-\frac{1}{2} - \al - \eps} r^{-\gamma} \norm*{v_{0, r} + \<1>_{0, r}}_{-\al} \dx{r} \lesssim \eta t^{-2\al - 2\eps}.
    \end{align}
\end{proof}

\section{Proof of the Bakry--Émery identity}\label{sec:BEI}

In \cite{RZZ17I} it was proved that
\begin{align}\label{dirichletphi}
    \mathcal{E}(F, F) := \int_{\mathcal{S}'(\T^2)} \norm*{DF}^2_{L^2_x} \dx\nu
\end{align}
where $F \in \mathcal{F}C^{\infty}_b$
is closable (see also \cite{AR91}) and the closure gives rise to a quasi-regular Dirichlet form (cf. \cite{MR92}), hence to a generator $\mathcal{L}$ with domain $D(\mathcal{L}) \subset D(\mathcal{E})$ such that
\begin{align}
    \mathcal{E}(F, F) = - \int_{\mathcal{S}'(\T^2)} F \mathcal{L}F \dx\nu.
\end{align}
We denote by $\{\tilde{P}_t\}_{t \geq 0}$ the associated semi-group. Then, by \cite[Theorem 3.13]{RZZ17I}, we infer that $P_tF = \tilde{P}_tF$ $\nu$-almost surely for all $F \in \mathcal{F}C^{\infty}_b$ and hence by continuity in time they are indistinguishable (see also \cite[p. 67]{Ka05}).
Moreover, by \cite[Theorem 3.7]{RZZ17I} (and the discussion thereafter) $K:=C^{\infty}(\T^2) \subset L^2(\T^2)$ is a a dense and linear subspace consisting of $\nu$-admissible elements. Hence assumptions (C.1), (C.2) and (C.3) of \cite[Section 4]{AR91} are fulfilled. Moreover, $f \mapsto P_tF(f)$ is quasi-continuous for any $F \in \mathcal{F}C^{\infty}_b$. Now we can prove \cref{BEtrick}.

\begin{proof}[Proof of \cref{BEtrick}]
    Following \cite[Proof of Theorem 1.1]{Ka06} we prove the $\nu$-a.s. identity 
    \begin{align}
        \frac{\dx}{\dx s} P_{t-s}\bra*{P_s F}^2 = - 2 P_{t-s} \bra*{\norm*{D P_s F}^2_{L^2_x}}.
    \end{align}
    and use the same notation. Let $0 \leq r_1, r_2 \leq t$ and define $H(r_1, r_2) := P_{t-r_1} \bra*{P_{r_2}F}^2$. 
    
    By \cite[p. 364, Theorem 4.3]{AR91} and since $P_{r_2}F \in D(\mathcal{E})$ it holds that
    \begin{align}
        P_{r_2}F(u_r^f) - P_{r_2}F(f) = \int_0^r \mathcal{L}(P_{r_2}F)(u_s^f) \dx{s} + M_r
    \end{align}
    where $M$ is a continuous martingale.
    
    Moreover, by \cite[p. 365, Proposition 4.5]{AR91} the quadratic variation of $M$ is given by
    \begin{align}
        \langle M \rangle_r = \int_0^r \norm*{DP_{r_2}F(u_s^f)}^2_{L^2_x} \dx{s}
    \end{align}
    Then by It\^o's formula \cite[p. 222, Theorem 3.3]{RY99} we compute
    \begin{align}
        \bra*{P_{r_2}F}^2(u_r^f) = \bra*{P_{r_2}F}^2(f) &+ 2\int_0^r P_{r_2}F(u_s^f) \dx M_s + 2 \int_0^r P_{r_2}F(u_s^f) \mathcal{L}(P_{r_2}F)(u_s^f) \dx{s} \\
        &+ 2 \int_0^r \norm{DP_{r_2}F (u_s^f)}^2_{L^2_x} \dx{s}
    \end{align}
    and hence
    \begin{align}
        P_{t - r_1}\bra*{P_{r_2}F}^2(f) = \bra*{P_{r_2}F}^2(f) &+ 2 \int_0^{t-r_1} P_{s}\bra*{P_{r_2}F \mathcal{L}P_{r_2}F}(f) \dx{s} 
        + 2 \int_0^{t - r_1} P_s\norm*{DP_{r_2}F}^2_{L^2_x}(f) \dx{s}.
    \end{align}
    Then we see that
    \begin{align}
        \frac{\partial}{\partial r_1} P_{t - r_1}\bra*{P_{r_2}F}^2(f) = -2 P_{t-r_1}\bra*{P_{r_2}F \mathcal{L}P_{r_2}F}(f) - 2 P_{t-r_1}\norm*{DP_{r_2}F}^2_{L^2_x}(f)
    \end{align}
    and on the other hand we have
    \begin{align}
        \frac{\partial}{\partial r_2} P_{t - r_1}\bra*{P_{r_2}F}^2(f) = 2 P_{t-r_1}\bra*{P_{r_2}F \mathcal{L}P_{r_2}F}(f).
    \end{align}
    Continuity follows in the same vain as in \cite[Proof of Theorem 1.1]{Ka06}.
    Finally, we have
    \begin{align}
         \frac{\dx}{\dx s} P_{t-s}\bra*{P_s F}^2 &= \frac{\partial}{\partial r_1} P_{t - r_1}\bra*{P_{r_2}F}^2(f)\Bigr|_{r_1 = r_2 = s} + \frac{\partial}{\partial r_2} P_{t - r_1}\bra*{P_{r_2}F}^2(f)\Bigr|_{r_1 = r_2 = s} \\ 
         &= - 2 P_{t-r_1}\norm*{DP_{r_2}F}^2_{L^2_x}(f).
    \end{align}
    Integrating from $0$ to $t$ proves the claim.
\end{proof}

\section{Differentiability with respect to the initial data}\label{diffvproof}

We set
\begin{align}
    G(f, v)(t) := S(t)f + \int_0^t S(t-s)F(v_s) \dx{s} - v_t
\end{align}
where $F(v_t) := -\bra*{v_t^3 + 3 v_t^2 \<1>_t + 3 v_t \<2>_t + \<3>_t - c_{t, \infty}(v_t + \<1>_t)}$. By \cite[Theorem 3.9]{TW18} 
there exist fixed parameters $\gamma, \beta > 0$ such that for any $f^* \in C^{-\al_0}$ and $T>0$ we can find a unique solution 
$v^*$ to \eqref{eq:remainder} satisfying $G(f^*, v^*)(t) = 0$, for every $0 \leq t \leq T$, and $\sup_{0 \leq t \leq T} (t\wedge1)^{\gamma} \norm*{v^*_t}_{\beta} < \infty$. We define
\begin{align}
    X := \set*{f \in C^{-\al_0} : \norm*{f}_{-\al_0} \leq R}, \quad Y := \set*{v: \pra*{0, T} \to C^{\beta} : \sup_{0 \leq t \leq T \wedge T^*} t^{\gamma} \norm*{v_t}_{\beta} \leq 1}
\end{align}
for some $T^*$ to be chosen below. Then again by \cite[Theorem 3.9]{TW18} we know that $G(f^*, v^*)(t) = 0$, for every $0 \leq t \leq T \wedge T^*$. It is easy to check that $G$ is Frechét-differentiable and we have
\begin{align}
    G_v(f^*, v^*)\delta v (t) = \int_0^t S(t-s) \bra*{F'(v_s) \delta v_s} \dx{s} - \delta v_t =: (K - Id)\delta v_t
\end{align}
where $F'(v_t) \delta v_t := -3\bra*{v_t^2 + 2 v_t \<1>_t +\<2>_t - c_{t, \infty}}\delta v_t$. A simple calculation shows that
\begin{align}
    \norm*{K\delta v_t}_{\beta} &\lesssim \int_0^t \bra*{t-s}^{-\frac{\al + \beta}{2}} s^{-\gamma} \norm*{\delta v_s}_{\beta} \dx{s} 
    \lesssim \bra*{T \wedge T^*}^{1 - \frac{\al + \beta}{2} - \gamma} \sup_{0 \leq t \leq T^* \wedge T'} t^{\gamma} \norm*{\delta v_t}_{\beta}.
\end{align}
Choosing $T^*$ small enough such that the r.h.s. above is strictly smaller than $1$ we get by the Neumann-series criterion 
that $G_v(f^*, v^*): Y \to Y $ is a bijection. Hence by \cite[Theorem 4.E]{Z95} we get that $f \mapsto v^f$ is differentiable and
its derivative in $h$ is a mild solution to \eqref{inieq} on $(0, T \wedge T*]$. Concatenating this argument to cover the whole time interval 
$(0,T]$ proves the assertion.

\end{appendices}

\addtocontents{toc}{\protect\setcounter{tocdepth}{1}}

\pdfbookmark{References}{references}
\addtocontents{toc}{\protect\contentsline{section}{References}{\thepage}{references.0}}

\bibliographystyle{plainurl}
\bibliography{refs}

\begin{thebibliography}{10}

\bibitem{AR91}
S.~Albeverio and M.~R\"{o}ckner.
\newblock Stochastic differential equations in infinite dimensions: solutions
  via {D}irichlet forms.
\newblock {\em Probab. Theory Related Fields}, 89(3):347--386, 1991.
\newblock \href {https://doi.org/10.1007/BF01198791}
  {\path{doi:10.1007/BF01198791}}.

\bibitem{B06}
D.~Bakry.
\newblock Functional inequalities for {M}arkov semigroups.
\newblock In {\em Probability measures on groups: recent directions and
  trends}, pages 91--147. Tata Inst. Fund. Res., Mumbai, 2006.

\bibitem{BE85}
D.~Bakry and M.~\'{E}mery.
\newblock Diffusions hypercontractives.
\newblock In {\em S\'{e}minaire de probabilit\'{e}s, {XIX}, 1983/84}, volume
  1123 of {\em Lecture Notes in Math.}, pages 177--206. Springer, Berlin, 1985.
\newblock \href {https://doi.org/10.1007/BFb0075847}
  {\path{doi:10.1007/BFb0075847}}.

\bibitem{BGL14}
D.~Bakry, I.~Gentil, and M.~Ledoux.
\newblock {\em Analysis and geometry of {M}arkov diffusion operators}, volume
  348 of {\em Grundlehren der Mathematischen Wissenschaften [Fundamental
  Principles of Mathematical Sciences]}.
\newblock Springer, Cham, 2014.
\newblock \href {https://doi.org/10.1007/978-3-319-00227-9}
  {\path{doi:10.1007/978-3-319-00227-9}}.

\bibitem{BD22}
R.~Bauerschmidt and B.~Dagallier.
\newblock Log-sobolev inequality for the $\varphi^4_2$ and $\varphi^4_3$
  measures, 2022.
\newblock \href {http://arxiv.org/abs/2202.02295} {\path{arXiv:2202.02295}}.

\bibitem{BB21}
T.~Bauerschmidt, R.and~Bodineau.
\newblock Log-{S}obolev inequality for the continuum sine-{G}ordon model.
\newblock {\em Comm. Pure Appl. Math.}, 74(10):2064--2113, 2021.
\newblock \href {https://doi.org/10.1002/cpa.21926}
  {\path{doi:10.1002/cpa.21926}}.

\bibitem{CLL13}
T.~Cass, C.~Litterer, and T.~Lyons.
\newblock Integrability and tail estimates for {G}aussian rough differential
  equations.
\newblock {\em Ann. Probab.}, 41(4):3026--3050, 2013.
\newblock \href {https://doi.org/10.1214/12-AOP821}
  {\path{doi:10.1214/12-AOP821}}.

\bibitem{CG14}
P.~Cattiaux and A.~Guillin.
\newblock Semi log-concave {M}arkov diffusions.
\newblock In {\em S\'{e}minaire de {P}robabilit\'{e}s {XLVI}}, volume 2123 of
  {\em Lecture Notes in Math.}, pages 231--292. Springer, Cham, 2014.
\newblock URL: \url{https://doi.org/10.1007/978-3-319-11970-0_9}, \href
  {https://doi.org/10.1007/978-3-319-11970-0\_9}
  {\path{doi:10.1007/978-3-319-11970-0\_9}}.

\bibitem{DPD03}
G.~Da~Prato and A.~Debussche.
\newblock Strong solutions to the stochastic quantization equations.
\newblock {\em Ann. Probab.}, 31(4):1900--1916, 2003.
\newblock \href {https://doi.org/10.1214/aop/1068646370}
  {\path{doi:10.1214/aop/1068646370}}.

\bibitem{DO20}
M.~Duerinckx and F.~Otto.
\newblock Higher-order pathwise theory of fluctuations in stochastic
  homogenization.
\newblock {\em Stoch. Partial Differ. Equ. Anal. Comput.}, 8(3):625--692, 2020.
\newblock \href {https://doi.org/10.1007/s40072-019-00156-4}
  {\path{doi:10.1007/s40072-019-00156-4}}.

\bibitem{EKS15}
M.~Erbar, K.~Kuwada, and K.T. Sturm.
\newblock On the equivalence of the entropic curvature-dimension condition and
  {B}ochner's inequality on metric measure spaces.
\newblock {\em Invent. Math.}, 201(3):993--1071, 2015.
\newblock \href {https://doi.org/10.1007/s00222-014-0563-7}
  {\path{doi:10.1007/s00222-014-0563-7}}.

\bibitem{GO11}
A.~Gloria and F.~Otto.
\newblock An optimal variance estimate in stochastic homogenization of discrete
  elliptic equations.
\newblock {\em Ann. Probab.}, 39(3):779--856, 2011.
\newblock \href {https://doi.org/10.1214/10-AOP571}
  {\path{doi:10.1214/10-AOP571}}.

\bibitem{GH19}
M.~Gubinelli and M.~Hofmanov\'{a}.
\newblock Global solutions to elliptic and parabolic {$\Phi^4$} models in
  {E}uclidean space.
\newblock {\em Comm. Math. Phys.}, 368(3):1201--1266, 2019.
\newblock \href {https://doi.org/10.1007/s00220-019-03398-4}
  {\path{doi:10.1007/s00220-019-03398-4}}.

\bibitem{GH21}
M.~Gubinelli and M.~Hofmanov\'{a}.
\newblock A {PDE} construction of the {E}uclidean {$\phi_3^4$} quantum field
  theory.
\newblock {\em Comm. Math. Phys.}, 384(1):1--75, 2021.
\newblock \href {https://doi.org/10.1007/s00220-021-04022-0}
  {\path{doi:10.1007/s00220-021-04022-0}}.

\bibitem{GIP15}
M.~Gubinelli, P.~Imkeller, and N.~Perkowski.
\newblock Paracontrolled distributions and singular {PDE}s.
\newblock {\em Forum Math. Pi}, 3:e6, 75, 2015.
\newblock \href {https://doi.org/10.1017/fmp.2015.2}
  {\path{doi:10.1017/fmp.2015.2}}.

\bibitem{Ha14}
M.~Hairer.
\newblock A theory of regularity structures.
\newblock {\em Invent. Math.}, 198(2):269--504, 2014.
\newblock \href {https://doi.org/10.1007/s00222-014-0505-4}
  {\path{doi:10.1007/s00222-014-0505-4}}.

\bibitem{HMK18}
M.~Hairer and K.~Matetski.
\newblock Discretisations of rough stochastic {PDE}s.
\newblock {\em Ann. Probab.}, 46(3):1651--1709, 2018.
\newblock \href {https://doi.org/10.1214/17-AOP1212}
  {\path{doi:10.1214/17-AOP1212}}.

\bibitem{HM18}
M.~Hairer and J.~Mattingly.
\newblock The strong {F}eller property for singular stochastic {PDE}s.
\newblock {\em Ann. Inst. Henri Poincar\'{e} Probab. Stat.}, 54(3):1314--1340,
  2018.
\newblock \href {https://doi.org/10.1214/17-AIHP840}
  {\path{doi:10.1214/17-AIHP840}}.

\bibitem{HS21}
M.~Hairer and P.~Sch\"{o}nbauer.
\newblock The support of singular stochastic partial differential equations.
\newblock {\em Forum Math. Pi}, 10:Paper No. e1, 127, 2022.
\newblock \href {https://doi.org/10.1017/fmp.2021.18}
  {\path{doi:10.1017/fmp.2021.18}}.

\bibitem{IORT20}
R.~Ignat, F.~Otto, T.~Ried, and P.~Tsatsoulis.
\newblock Variational methods for a singular spde yielding the universality of
  the magnetization ripple, 2020.
\newblock \href {http://arxiv.org/abs/2010.13123} {\path{arXiv:2010.13123}}.

\bibitem{Ka05}
H.~Kawabi.
\newblock The parabolic {H}arnack inequality for the time dependent
  {G}inzburg-{L}andau type {SPDE} and its application.
\newblock {\em Potential Anal.}, 22(1):61--84, 2005.
\newblock \href {https://doi.org/10.1007/s11118-004-6456-4}
  {\path{doi:10.1007/s11118-004-6456-4}}.

\bibitem{Ka06}
H.~Kawabi.
\newblock A simple proof of log-{S}obolev inequalities on a path space with
  {G}ibbs measures.
\newblock {\em Infin. Dimens. Anal. Quantum Probab. Relat. Top.},
  9(2):321--329, 2006.
\newblock \href {https://doi.org/10.1142/S021902570600238X}
  {\path{doi:10.1142/S021902570600238X}}.

\bibitem{LOTT21}
P.~Linares, F.~Otto, M.~Tempelmayr, and P.~Tsatsoulis.
\newblock A diagram-free approach to the stochastic estimates in regularity
  structures, 2021.
\newblock \href {http://arxiv.org/abs/2112.10739} {\path{arXiv:2112.10739}}.

\bibitem{MR92}
Z.~M. Ma and M.~R\"{o}ckner.
\newblock {\em Introduction to the theory of (nonsymmetric) {D}irichlet forms}.
\newblock Universitext. Springer-Verlag, Berlin, 1992.
\newblock \href {https://doi.org/10.1007/978-3-642-77739-4}
  {\path{doi:10.1007/978-3-642-77739-4}}.

\bibitem{MW20}
A.~Moinat and H.~Weber.
\newblock Space-time localisation for the dynamic {$\Phi^4_3$} model.
\newblock {\em Comm. Pure Appl. Math.}, 73(12):2519--2555, 2020.
\newblock \href {https://doi.org/10.1002/cpa.21925}
  {\path{doi:10.1002/cpa.21925}}.

\bibitem{MW17II}
J.-C. Mourrat and H.~Weber.
\newblock The dynamic {$\Phi^4_3$} model comes down from infinity.
\newblock {\em Comm. Math. Phys.}, 356(3):673--753, 2017.
\newblock \href {https://doi.org/10.1007/s00220-017-2997-4}
  {\path{doi:10.1007/s00220-017-2997-4}}.

\bibitem{MW17I}
J.-C. Mourrat and H.~Weber.
\newblock Global well-posedness of the dynamic {$\Phi^4$} model in the plane.
\newblock {\em Ann. Probab.}, 45(4):2398--2476, 2017.
\newblock \href {https://doi.org/10.1214/16-AOP1116}
  {\path{doi:10.1214/16-AOP1116}}.

\bibitem{Ne}
E.~Nelson.
\newblock The free {M}arkoff field.
\newblock {\em J. Functional Analysis}, 12:211--227, 1973.
\newblock \href {https://doi.org/10.1016/0022-1236(73)90025-6}
  {\path{doi:10.1016/0022-1236(73)90025-6}}.

\bibitem{O01}
F.~Otto.
\newblock The geometry of dissipative evolution equations: the porous medium
  equation.
\newblock {\em Comm. Partial Differential Equations}, 26(1-2):101--174, 2001.
\newblock \href {https://doi.org/10.1081/PDE-100002243}
  {\path{doi:10.1081/PDE-100002243}}.

\bibitem{OV00}
F.~Otto and C.~Villani.
\newblock Generalization of an inequality by {T}alagrand and links with the
  logarithmic {S}obolev inequality.
\newblock {\em J. Funct. Anal.}, 173(2):361--400, 2000.
\newblock \href {https://doi.org/10.1006/jfan.1999.3557}
  {\path{doi:10.1006/jfan.1999.3557}}.

\bibitem{PW81}
G.~Parisi and Y.~S. Wu.
\newblock Perturbation theory without gauge fixing.
\newblock {\em Sci. Sinica}, 24(4):483--496, 1981.

\bibitem{RY99}
D.~Revuz and M.~Yor.
\newblock {\em Continuous martingales and {B}rownian motion}, volume 293 of
  {\em Grundlehren der Mathematischen Wissenschaften [Fundamental Principles of
  Mathematical Sciences]}.
\newblock Springer-Verlag, Berlin, third edition, 1999.
\newblock \href {https://doi.org/10.1007/978-3-662-06400-9}
  {\path{doi:10.1007/978-3-662-06400-9}}.

\bibitem{RZZ17II}
M.~R\"{o}ckner, R.~Zhu, and X.~Zhu.
\newblock Ergodicity for the stochastic quantization problems on the
  2{D}-torus.
\newblock {\em Comm. Math. Phys.}, 352(3):1061--1090, 2017.
\newblock \href {https://doi.org/10.1007/s00220-017-2865-2}
  {\path{doi:10.1007/s00220-017-2865-2}}.

\bibitem{RZZ17I}
M.~R\"{o}ckner, R.~Zhu, and X.~Zhu.
\newblock Restricted {M}arkov uniqueness for the stochastic quantization of
  {$P(\Phi)_2$} and its applications.
\newblock {\em J. Funct. Anal.}, 272(10):4263--4303, 2017.
\newblock \href {https://doi.org/10.1016/j.jfa.2017.01.023}
  {\path{doi:10.1016/j.jfa.2017.01.023}}.

\bibitem{TW18}
P.~Tsatsoulis and H.~Weber.
\newblock Spectral gap for the stochastic quantization equation on the
  2-dimensional torus.
\newblock {\em Ann. Inst. Henri Poincar\'{e} Probab. Stat.}, 54(3):1204--1249,
  2018.
\newblock \href {https://doi.org/10.1214/17-AIHP837}
  {\path{doi:10.1214/17-AIHP837}}.

\bibitem{TW20}
P.~Tsatsoulis and H.~Weber.
\newblock Exponential loss of memory for the 2-dimensional {A}llen-{C}ahn
  equation with small noise.
\newblock {\em Probab. Theory Related Fields}, 177(1-2):257--322, 2020.
\newblock \href {https://doi.org/10.1007/s00440-019-00945-x}
  {\path{doi:10.1007/s00440-019-00945-x}}.

\bibitem{RS05}
M.K. von Renesse and K.T. Sturm.
\newblock Transport inequalities, gradient estimates, entropy, and {R}icci
  curvature.
\newblock {\em Comm. Pure Appl. Math.}, 58(7):923--940, 2005.
\newblock \href {https://doi.org/10.1002/cpa.20060}
  {\path{doi:10.1002/cpa.20060}}.

\bibitem{Z95}
E.~Zeidler.
\newblock {\em Applied functional analysis}, volume 109 of {\em Applied
  Mathematical Sciences}.
\newblock Springer-Verlag, New York, 1995.
\newblock Main principles and their applications.

\end{thebibliography}

\begin{flushleft}
\footnotesize \normalfont
\textsc{Florian Kunick\\
Max--Planck--Institute for Mathematics in the Sciences\\ 
04103 Leipzig, Germany}\\
\texttt{\textbf{florian.kunick@mis.mpg.de}}
\end{flushleft}

\begin{flushleft}
\footnotesize \normalfont
\textsc{Pavlos Tsatsoulis\\
Faculty of Mathematics, University of Bielefeld\\
33615 Bielefeld, Germany}\\
\texttt{\textbf{ptsatsoulis@math.uni-bielefeld.de}}
\end{flushleft}

\end{document}